\documentclass[10pt]{amsart}

\usepackage{geometry,enumitem}
\usepackage{hyperref}
\usepackage{graphicx}
\usepackage[usenames, dvipsnames]{color}
\usepackage{amsthm, amsxtra}

\newenvironment{myitemize}
{ \begin{itemize}
    \setlength{\itemsep}{0pt}
    \setlength{\parskip}{0pt}
    \setlength{\parsep}{0pt} 
     \setlength{\itemindent}{-20pt} }     
{ \end{itemize}                  }

\numberwithin{equation}{section}

\newtheorem{theorem}{Theorem}[section]
\newtheorem{theorem*}{Theorem}
\newtheorem{remark}[theorem]{Remark}
\newtheorem{lemma}[theorem]{Lemma}

\newtheorem{proposition}[theorem]{Proposition}
\newtheorem{corollary}[theorem]{Corollary}

\newtheorem*{question*}{Question}

\newtheorem{example}[theorem]{Example}

\newtheorem{definition}[theorem]{Definition}

\newcommand{\nn}{\nonumber}
\newcommand{\R}{\mathbb{R}}
\newcommand{\E}{\mathbb{E}}

\author{Nassif Ghoussoub, Young-Heon Kim and Aaron Zeff Palmer}

\address{* Nassif Ghoussoub, Young-Heon Kim, and Aaron Zeff Palmer}
\address{Department of Mathematics\\ University of British Columbia\\ Vancouver, V6T 1Z2 Canada}
\email{nassif@math.ubc.ca, yhkim@math.ubc.ca,  azp@math.ubc.ca}

\title[Optimal Stopping of Stochastic Transport]{Optimal Stopping of Stochastic Transport Minimizing Submartingale Costs}

\thanks{The  first two 
authors are partially supported by  the 
Natural Sciences and Engineering Research Council of Canada (NSERC). \\
\copyright 2020 by the author.}

\date{\today}

\begin{document}

\begin{abstract}
	Given a stochastic state process $(X_t)_t$ and a real-valued submartingale cost process $(S_t)_t$, we characterize optimal stopping times $\tau$ that minimize the expectation of $S_\tau$ while realizing given initial and target distributions $\mu$ and $\nu$, i.e., $X_0\sim \mu$ and $X_\tau \sim \nu$. A dual  optimization problem is considered and shown to be attained under suitable conditions.  The optimal solution of the dual problem then provides a contact set, which characterizes the location where optimal stopping can occur. The optimal stopping time is uniquely determined as the first hitting time of this contact set provided we assume a natural structural assumption on the pair $(X_t, S_t)_t$, which generalizes  the {\it twist condition} on the cost in optimal transport theory.  This paper extends the Brownian motion settings studied in \cite{GKPS,GKP-Monge} and deals with more general costs. 
\end{abstract}

\maketitle
\tableofcontents

\section{Introduction} \label{sec:introduction}

	Given a state process $(X_t)_t$ valued in a complete metric space $O$, an initial distribution $\mu$, and a target distribution $\nu$ on $O$, we consider the set $\mathcal{T}(\mu,\nu)$ of -possibly randomized- stopping times $\tau$ that satisfy 
				\begin{align*} 
		X_0\sim \mu\ \ \ {\rm and}\ \ \ X_\tau\sim\nu, 
			\end{align*}
where here, and in the sequel, the notation $Y\sim \lambda$ means that the law of the random variable $Y$ is the probability measure $\lambda$.  The problem of finding such stopping times (i.e., when $\mathcal{T}(\mu,\nu)$ is non-empty) is known as the {\it Skorokhod embedding problem} and has a long history ever since it was initiated by Skorokhod \cite{skorokhod1965studies} in the early 1960s in the case where $(X_t)_t$ is Brownian motion $(W_t)_t$ and $O=\R$, and followed by important contributions from Root \cite{root1969existence}, Rost \cite{rost1971stopping}, Chacon-Walsh \cite{chacon-walsh76} and others. In this case, the set $\mathcal{T}(\mu,\nu)$ of embedding stopping times is empty unless $\mu$ and $\nu$ satisfies a certain order:
	\begin{align*}
 		\mu \prec \nu \quad {\rm 
	that\, is}\quad  
		\int_{\R} \phi(x)\mu(dx) \le \int_{\R} \phi(y) \nu(dy),\ \hbox{\rm for all subharmonic functions } \phi \hbox{ ($\Delta \phi \ge 0$)}. 
	\end{align*}
	This condition is indeed sufficient and illustrates a duality principle of embedding stopping times, which is given here as Corollary \ref{cor:Strassen} to the duality Theorem \ref{thm:weak_duality} (see \cite{beiglboeck2017optimal}   as well as \cite{GKL-radial}). Since then, the problem and its variants were investigated by a large number of researchers, and have  led to several important results in  
probability theory and stochastic processes. We refer to Ob\l oj \cite{obloj2004skorokhod} for an excellent survey of the subject.   
	
For our present purpose, we note that no special role is played here by $O=\R$ or by $(X_t)_t$ being Brownian motion, and the Skorokhod embedding theorem was eventually extended to more general state spaces and Markov processes (see the previously related works of Baxter-Chacon \cite{baxter1976stopping}, Falkner \cite{falkner1980skorohod},  Strassen \cite{strassen1965existence}, Rost \cite{rost1971stopping}, Dellacherie-Meyer \cite{Dellacherie-Meyer-78}, etc). For example, the result holds generally when $\Delta$ is the infinitesimal generator of a diffusion process $(X_t)_t$.  Our analysis here will distinguish between two cases: processes that are absorbed into a `cemetery' state, and those that are ergodic.  We shall detail the absorbing case in the main body of the paper and repeat the results for ergodic processes in Appendix \ref{sec:recurrent}.

Given now a real valued cost process $(S_t)_t$, {\it the optimal Skorokhod embedding problem} is to minimize the expected cost over all such embedding stopping times, cf.\ \cite{beiglboeck2017optimal},
	\begin{align}\label{eqn:primal}
		\mathcal{P}_S(\mu,\nu): = \inf_{\tau}\Big\{\mathbb{E}\big[S_\tau\big];\ X_0\sim \mu,\ X_\tau\sim \nu\Big\}.
	\end{align}

	The optimization problem \eqref{eqn:primal} and its variants have been considered in mathematical finance, for example, with applications to option-pricing by Hobson \cite{hobson2011skorokhod}, Beiglb\"{o}ck and Juillet \cite{beiglboeck-juillet2016}, Ghoussoub-Kim-Lim \cite{ghoussoub2015structure, GKL-radial}, Beiglb\"{o}ck-Cox-Huesmann \cite{beiglboeck2017optimal}.  Many applications consider processes beyond Brownian motion, and the cost processes possessing a variety of structural properties.  A starting point for our work is that the cost should be a submartingale, in other words, a process that increases in conditional expectation for each increment. There are two particular cases of submartingale costs that have already been analyzed when the state process $X_t=W_t$ is a multi-dimensional Brownian motion:
		\begin{myitemize}
			\item ${ S}_t = \int_0^t {  L}(s,W_s)ds$,  where the Lagrangian $L$ is non-negative was analyzed in \cite{GKPS,GKP-Monge};
			\item ${ S}_t = { c}(W_0,W_t)$  where $y\to c(x, y)$ is subharmonic, i.e., $\Delta_yc(x,y) \geq 0$, which was analyzed in \cite{GKP-Monge} 
		\end{myitemize}
Note that the first case is Markovian, while the second is not though it depends only on initial/final position. 

The dual problem to $\mathcal{P}_S(\mu, \nu)$ has been expressed in \cite{beiglboeck2017optimal}, as
	\begin{align}\label{eqn:dual}
		\mathcal{D}_S(\mu,\nu) = \sup_{(\psi,(M_t)_t)\in \mathcal{A}_S}\Big\{\int_{{O}}\psi(y)\nu(dy) - \mathbb{E}^{\mathbb{P}^\mu}\big[M_0\big]\Big\},
	\end{align}
	where $\mathbb{P}^\mu$ is the distribution of $(X_t)_t$ with initial $X_0\sim \mu$, and $\mathcal{A}_S$ consists of the end potential, $\psi:O\rightarrow \R$, and a martingale $(M_t)_t$ that satisfies $M_\sigma\geq \psi(X_\sigma)- S_\sigma$ for any stopping time $\sigma$ almost surely on the probability space.  

	In general, the dual maximization problem can be reduced to a maximization of $\psi$, as $M_0$ can be determined from $\psi$ using the Snell envelope of $(\psi(X_t)-S_t)_t$, which we denote by $(G^\psi_t)_t$.  In other words, $G^\psi_t$ is the conditional expected value of the optimal stopping problem that maximizes $\psi(X_\sigma)-S_\sigma$ over stopping times $\sigma\geq t$.  The martingale $(M^\psi_t)_t$ such that $(\psi,(M^\psi_t)_t)\in \mathcal{A}_S$ can be recovered from $(G^\psi_t)_t$ by the Doob-Meyer decomposition. These results are covered by Lemma \ref{lem:value_process} and Theorem \ref{thm:weak_duality}.

	The optimizer $\psi$ of the dual problem characterizes an optimal embedding stopping time, $\tau\in \mathcal{T}(\mu,\nu)$, by the following `verification' principles of Theorem \ref{thm:verification}:
	\begin{enumerate}[label=\roman*.]
		\item The value process does not change predictably before $\tau$ ($(G_{t\wedge \tau}^\psi)_t$ is a martingale);
		\item The terminal value is given by the end potential minus the cost ($G^\psi_{\tau}=\psi(X_{\tau})-S_{\tau}$).
	\end{enumerate}

	The first focus of this paper is on the attainment of the dual maximization problems, $\mathcal{D}_{S}(\mu, \nu)$, which so far has been elusive. In one-dimension, dual attainment has been achieved for Brownian motion in \cite{beiglbock2017complete} and \cite{HLT2015}, and has been extended to the multi-dimensional case in \cite{GKPS} and \cite{GKP-Monge}.  We also note an alternate approach to attainment has been undertaken by weakening the dual formulation in \cite{beiglbock2019fine}. Our analysis identifies natural structural situations that yield bounds, hence compactness in suitable function spaces, on the end potential $\psi$. Some structure is required, since if the cost is (the supermartingale!) $S_t = -|W_0-W_t|$, then the dual maximizer $\psi$ does not exist \cite{beiglboeck-juillet2016} (see also \cite{ghoussoub2015structure} for related results). 

	In \S \ref{sec:pointwise_bounds} we prove bounds on $\psi$ under the following set of assumptions:
	\begin{myitemize}
		\item The state process $(X_t)_t$ is a stationary Feller process with an absorbing state $\mathfrak{C}\in O$;
		\item The initial and target distribution are in the order prescribed by $(X_t)_t$, $\mu \prec \nu$, i.e.\
		$$
			\int_{O} \phi(x)\mu(dx)\leq \int_O \phi(y)\nu(dy), \mbox{whenever}\ \phi \ \mbox{is $X$-subharmonic (or $\Delta \phi(x)\geq 0$)},
		$$
		where we abuse the notation and let $\Delta$ denote the generator of $X$.
		\item  $\mathbb{E}[\tau_\mathfrak{C}]<\infty$ for $\tau_\mathfrak{C} = \inf\{ t \ | \ X_t = \mathfrak{C} \}$ and $S_{\tau}=S_{\tau_\mathfrak{C}}$ for $\tau\geq \tau_\mathfrak{C}$.  (This is naturally satisfied when the process is either killed exponentially or by an absorbed state).
		\item The cost process $(S_t)_t$ is a submartingale with $S_0=0$ and $S_{\tau_\mathfrak{C}}\leq K<+\infty$.
	\end{myitemize} 
	Proposition \ref{prop:psi_normalization} provides the essential normalization for absorbing processes, which shows that the end potential $\psi$ can be restricted to values in $[-K,0]$.  This procedure also restricts $\psi$ to have $0$ value on $\mathfrak{C}$ and any absorbing state of the Markov process.
These bounds are sufficient for dual attainment if $O$ is discrete (Theorem \ref{thm:finite_attainment}).  To illustrate the generality of our approach, we shall also handle the case where  the time is discrete.

In \S \ref{sec:dual_attainment}, we aim for a more refined bound and make the following assumptions:
	\begin{myitemize}
		\item The state process $(X_t)_t$ is symmetric with Dirichlet form $\mathcal{E}$ that satisfies a Poincar\'{e} inequality, 
		$$ 
			\|f\|_2^2:= \int_O |f(x)|^2m(dx) \leq 
C_p\|f\|_{\mathcal{H}}^2, 
		$$
		where $\|f\|_{\mathcal{H}}^2:=\mathcal{E}(f,f)=\int_O f(x)\big(-\Delta f(x)\big)m(dx)$.
	\item We also assume $\mathcal{E}$ satisfies a regularity property that makes the viscosity and variational formulations of supersolutions equivalent.
		\item The initial distribution $\mu$ lies in the dual space $\mathcal{H}^*$;
		\item The cost process $(S_t)_t$ is such that $(Dt-S_t)_t$ is also a submartingale for some $D>0$. In other words, 
		$$
			\mathbb{E}\big[S_r|\mathcal{F}_t]-S_t\leq D(r-t)\,\hbox{ for $r>t$.}
		$$
			\end{myitemize}

	We then obtain a uniform superharmonic estimate on the end potentials of the form
	$$
		\Delta \psi(x) \leq D \quad \hbox{for all $x\in O$},
	$$
	which translates to a uniform bound of $\|\psi\|_\mathcal{H}$,
	by combining Lemma \ref{lem:psi_maximization} with viscosity and weak solution theory in Proposition \ref{prop:dual_normalization}, and prove semi-continuity of the dual value with respect to the topology of $\mathcal{B}_D$ in Proposition \ref{prop:compact_semicontinuous}.
	 With these results, we prove attainment of optimal $\psi^*$ in the Hilbert space.

	 For ergodic processes the situation is different. The dual potentials are no longer bounded above and are no longer limited to the Hilbert space $\mathcal{H}$.  We still however prove dual attainment in a suitable space $\mathcal{B}_D'$, where the truncated potentials $\min\{\psi(x),M\}$ live in the Hilbert space for each $M>0$; see Theorem \ref{thm:strong_dual_recurrent}.

	 The differences between the absorbing and ergodic processes can be better understood by examining the case of minimizing the expectation of embedding stopping times.  In the absorbing case,  the expectation of the stopping time is determined solely by $\mu$ and $\nu$ as long as it has finite expectation (Proposition \ref{prop:expected_time_transient}). For ergodic processes, the minimum expected time is given by a remarkable duality; see Theorem \ref{thm:ergodic_duality}.  The dual is obtained as the potential of a point mass (which does not belong to the underlying Hilbert space), which characterizes the optimal stopping times by requiring that their local time at that point is zero (Corollary \ref{cor:Necessary_sufficient}).

	Finally, a novel result of \cite{GKP-Monge} was to identify a {\it stochastic twist condition} on cost processes of the form  $S_t=c(X_0,X_t)$, which guarantees that the optimal stopping time is a hitting time of a barrier in the product space of the current and initial position.  We provide in Section~\ref{sec:generalTwist}  a new and more general twist condition on the pair $(S_t, X_t)$ to consolidate these previous results.  We suppose that the cost decomposes as $S_t=\Lambda(A_t,X_t)$, where $(A,X)$ is a $U\times O$-valued stationary Feller process, $U$ being an auxiliary differentiable manifold, and $\Lambda: U\times O \to \R$ is measurable and differentiable in the first variable.  Then, we say that the cost is {\it $(A,X)$-twisted} if a stopping time $\sigma$  is $0$ whenever
	$$
		\mathbb{E}^{a,x}\big[\nabla_a \Lambda(A_\sigma,X_\sigma)\big]=\nabla_a\Lambda (a,x)\quad \hbox{for any $a\in U$ and any $x\in O$}.	$$
We note how the previously known examples fall into this class. For example the Root and Rost embeddings are optimizers if $A_t=t$ and the cost is an increasing, strictly convex, or concave function of time.  The Monge costs considered in \cite{GKP-Monge}, and the stochastic twist condition considered there also fits in this context with $A_t=X_0$.  Further generalizations of the Root embedding of \cite{gassiat2019free} can also be considered here.  We shall prove that a unique optimizer is then given by the hitting time of a barrier in the space $U\times O$, given an additional regularity assumption on the processes and possibly $\mu$ and $\nu$.  Under these assumptions, we have that the stopping time is unique and given by the hitting time of a set in the product space of $U\times O$, which is determined by the dual problem.

\section{Weak Duality and Dynamic Programming}\label{sec:duality_dynamic_programming}

\subsection{Notation and Definitions}
	We mostly follow the general formulation of \cite{Guo-Tan-Touzi-2016monotone} and introduce more detailed assumptions later. 
	We let $\R^+$ be the nonnegative real numbers, $O$ be complete metric space, and  $O_\mathfrak{C}=O\cup \{\mathfrak{C}\}$ be the space with a cemetery state $\mathfrak{C}$ that is distance 1 from all other points. We suppose $(\Omega,\mathcal{F},(\mathcal{F}_t)_{t\in \R^+},\mathbb{P}^\mu)=(\Omega,\mathbb{F},\mathbb{P}^\mu)$ is a filtered probability space with $\Omega$ a Polish space and $\mathbb{P}^\mu$ a Borel measure such that $X:\Omega \rightarrow D(\R^+;O_\mathfrak{C})$ is a continuous map onto the Skorokhod space of c\'{a}dl\'{a}g paths equipped with the Skorokhod path space metric, i.e.\ the paths $(X_t(\omega))_{t\in \R^+}$ are continuous from the right with limits from the left.
	We suppose the filtration $\mathbb{F}$ is right continuous and $\mathcal{F}_0$ contains the null sets of $\mathbb{P}^\mu$.
	  We suppose that $X$ is adapted to $\mathbb{F}$.  

	  For a probability measure $\gamma$ on a Polish space $A$ and $\gamma$-integrable $f$, we use the notation
	$$
		\mathbb{E}^\gamma[f]=\int_A f(a)\gamma(da)
	$$
	and for a $\sigma$-algebra $\Sigma$ contained in the $\sigma$-algebra of $\gamma$-measurable sets, we take the conditional expectation
	$
		\mathbb{E}^\gamma[f|\Sigma]
	$
	to be the unique $\Sigma$-measurable function on $A$ such that
	$$
		\int_B\mathbb{E}^\gamma[f|\Sigma](a)\gamma(da)=\int_Bf(a)\gamma(da)
	$$
	for all sets $B\in \Sigma$.

	We let 
	$$\tau_\mathfrak{C} = \inf\{t; X_t = \mathfrak{C}\}$$
	denote the killing time of the process. We suppose that $\mathbb{E}^{\mathbb{P}^\mu}[\tau_\mathfrak{C}]<\infty$, and for any $\tau\in \R^+$,  $X_{\tau}= \mathfrak{C}$ holds for $\mathbb{P}^\mu$-a.e.\ $\omega$ such that $\tau\geq \tau_\mathfrak{C}(\omega)$, i.e.\ the path remains in the cemetery state after being killed with probability one.  We let	
\begin{align*}
 \hbox{	$\mathcal{S}$ denote the set of (nonrandomized) stopping times }
\end{align*}
and suppose also that for every open set $Q\subset O$ there exists $\sigma\in \mathcal{S}$ such that 
	$$
		\E^{\mathbb{P}^\mu}\big[\mathbf{1}\{X_\sigma\in Q\}\big]>0.
	$$
	We suppose the source distribution, $\mu$ a Borel probability measure on $O$, is the initial marginal of $\mathbb{P}^\mu$, i.e.\ 
	$$\mu = {X_0}_\# \mathbb{P}^\mu.$$
	We consider a target distribution $\nu$, a Borel probability measure on $O_\mathfrak{C}$, and the set of randomized stopping times $\mathcal{T}(\mu,\nu)$ that embed $\nu$ into the process $X$.  These correspond to measures $\bar{\mathbb{P}}$ on $\bar{\Omega}=\R^+\times \Omega$ (we let $\bar{\omega}=(\tau,\omega)$, $T(\bar{\omega})=\tau$, and $X_T(\bar{\omega})=X_{\tau}(\omega)$, and $\bar{\mathbb{F}}$ be the extended filtration with $\bar{\mathcal{F}}_t$ containing $B\times \Omega$ for Borel subsets $B\subset [0,t]$), such that 
	\begin{enumerate}[label=\roman*.]
		\item 
		 ${\pi^\Omega}_\#\bar{\mathbb{P}}=\mathbb{P}^\mu$ for the projection $\pi^\Omega(\bar{\omega})=\omega$; 
		\item    The process $\sigma\mapsto g(X_{\sigma})$  is uniformly integrable over all stopping times $\sigma\in \mathcal{S}$ for all $g\in C_b(O_\mathfrak{C})$.
		\item We assume that $T\leq \tau_\mathfrak{C}$ holds with probability 1.
 		\item ${X_T}_\# \bar{\mathbb{P}} = \nu$, equivalently $X_T\sim_{\bar{\mathbb{P}}} \nu$. 
	\end{enumerate}
	We will always consider the topology given by weak* convergence of the distribution of $T$ on $\R^+$ and the distribution of $X_T$ on $O_\mathfrak{C}$, for which $\mathcal{T}(\mu,\nu)$ is compact \cite{baxter1977compactness}.

	We also denote by $\mathcal{T}(\mu)$ the randomized stopping times with initial distribution $\mu$ and free stopping distribution, i.e.\ satisfying i., ii.,\ and iii.\ but not necessarily iv., and $\mathcal{T}_t(\mu)$ the randomized stopping times in $\mathcal{T}(\mu)$ with $T\geq t\wedge \tau_\mathfrak{C}$ almost surely. 

	We let $LSC_b(O_\mathfrak{C})$ be the set of bounded lower-semicontinuous functions on $O_\mathfrak{C}$, and we let 
\begin{itemize}
\item $B(\bar{\Omega})$ to be the processes that are jointly measurable with the Borel $\sigma$ algebra on $\R^+$ and the Borel $\sigma$ algebra completed with $\mathbb{P}^\mu$ null sets on $\Omega$, and uniformly integrable with respect to $\mathbb{P^\mu}$ for all stopping times. 

\end{itemize}

	We always assume the cost process $S\in B(\bar{\Omega})$  is adapted to the filtration  $\mathcal{\bar F}$. Even if $S$ were not adapted, it would not change the problem by replacing $S$ with its optional projection with respect to the filtration.

\begin{remark}
		All of our results can be easily adapted to discrete time, where the statements and proofs are the same with less technicalities. In Appendix {\normalfont\ref{sec:recurrent}}, we see that the uniform integrability can be easily relaxed to allow for unbounded costs.
	\end{remark}

\subsection{Dual formulation}

	We define
	\begin{align}\label{eqn:dual_admissible}
		\mathcal{A}_S = \big\{&(\psi,M)\in C_b(O_\mathfrak{C})\times B(\bar{\Omega});\\
		&\ \psi(X_\sigma)-M_\sigma\leq S_\sigma\ \forall\ \sigma\in \mathcal{S}\ {\rm and}\ \mathbb{P}^\mu-a.e.\ \omega,\nn\\  \nonumber
		&\ M\ {\rm is\ a\ }(\Omega,\mathbb{F},\mathbb{P}^\mu){\rm-martingale}\big\}.
	\end{align}
	The martingale condition is simply
	$$
		\E^{\mathbb{P}^\mu} \big[M_\sigma\big|\mathcal{F}_s\big]= M_s,\ \mathbb{P}^\mu-{\rm a.s.}
	$$
	for all stopping times $$\sigma\in \mathcal{S}_s=\big\{\sigma\in \mathcal{S};\ \sigma\geq s\big\}.$$

	The dual problem is 
	\begin{align}\label{eqn:dualproblem}
		\mathcal{D}_S(\mu,\nu) = \sup_{(\psi,M)\in \mathcal{A}_S} \Big\{\int_{O} \psi(y)\nu(dy)-\E^{\mathbb{P}^\mu}\big[M_0\big]\Big\}.
	\end{align}
	We note that if $X$ is Markov, then the set $\mathcal{A}_S$ does not depend on the initial distribution $\mu$, which will only appear in the cost of (\ref{eqn:dualproblem}).

	We consider the following assumptions to make the optimization problem well posed. First, we require lower-semicontinuity 
	\begin{enumerate} [label=\textbf{(A\arabic*)}] \setcounter{enumi}{-1}
		\item \label{itm:lower_semicontinuity} We suppose that $t\mapsto S_t$ is right-lower-semicontinuous for $\mathbb{P}^\mu$-a.e.\ $\omega$, and its predictable projection is left-lower-semicontinuous, and $S$ is bounded below.
		\item \label{itm:bounded_stopping} We suppose that $S_t=S_{\tau_\mathfrak{C}}$ for $t> \tau_\mathfrak{C}$ holds $\mathbb{P}^\mu$-a.s.\ and $S$ is class D (i.e.\ uniformly integrable over stopping times, which implies $\mathbb{E}^{\mathbb{P}^\mu}[S_{\tau_\mathfrak{C}}]<\infty$).
	\end{enumerate}

	The assumption \ref{itm:lower_semicontinuity} is made in \cite{bismut1979potential} (with upper semicontinuity) and can be considered as a lower-semicontinuous version of c\'{a}dl\'{a}g processes. The  assumption \ref{itm:bounded_stopping} is standard and includes two important general cases and restricts us to a compact set of stopping times, $T\wedge  \tau_\mathfrak{C}$.  One case is that the domain is an open set of a larger space $O\subset \bar{O}$ for which $\tau_\mathfrak{C}$ is the exit time, at which time the process enters the cemetery state $\mathfrak{C}$.   In some cases the problem on  $\bar{O}$ can be reduced to this by noting that there are states that cannot be reached with finite cost; see \cite{GKPS} where $X_t$ is Brownian motion and $O\subset \R^d$ is a bounded convex set.

	The second case is that the process is killed exponentially at rate $\beta>0$. We let $\hat{\mathbb{P}}^\mu$ be the distribution of the process on $O$ without killing. Then we have
	$$
		\mathbb{E}^{\mathbb{P}^\mu}\big[\mathbf{1}\{X_t\in O\}A_t\big] = e^{-\beta t} \mathbb{E}^{\hat{\mathbb{P}}^\mu}\big[A_t\big],
	$$
	for all processes $A_t$, and thus
	$$
		\mathbb{E}\big[\tau_\mathfrak{C}]= \int_{\R^+} t\beta e^{-\beta t} dt=\beta^{-1}<+\infty.
	$$
	The cost $S_t$  is uniformly integrable with respect to $\mathbb{P}^\mu$ if $e^{-\beta t}S_t$ is uniformly integrable with respect to $\hat{\mathbb{P}}^\mu$.
	The admissible target measures will be restricted as the constraint becomes for $f\in C_b(O_\mathfrak{C})$ with $f(\mathfrak{C})=0$,
	$$
		\mathbb{E}^{\bar{\mathbb{P}}}\Big[f\big(X_T\big)\Big]=\mathbb{E}^{\hat{\bar{\mathbb{P}}}}\Big[e^{-\beta T} f\big(X_T\big)\Big]=\int_O f(x)\nu(dx),
	$$
	and in particular if $\mu\not=\nu$ we require $\int_O \nu(dx)<1$ (equivalently, $\nu(\{\mathfrak{C}\})>0$).

	\subsection{Dynamic programming dual formulation}
	This section  employs the dynamic programming principle in a manner analogous to the double `convexification' procedure of optimal transportation.
	We first  note that the problem $\mathcal{D}_S(\mu,\nu)$ may be reduced to a concave maximization problem of $\psi$.  We
	define the value process $G^\psi$ to be the Snell envelope of $(\psi(X_t)-S_t)_{t\in \R^+}$, given by
	\begin{align}\label{eqn:value_process}
		G_t^\psi(\omega) :=&\ \sup_{\sigma\in \mathcal{S}_t}\Big\{ \mathbb{E}^{\mathbb{P}^\mu}\big[\psi(X_\sigma) - S_\sigma\big|{\mathcal{F}}_t\big](t,\omega)\Big\}\\
		=&\  \sup_{\bar{\mathbb{P}}\in \mathcal{T}_t(\mu)}\Big\{ \mathbb{E}^{\bar{\mathbb{P}}}\big[\psi(X_T) - S_T\big|\bar{\mathcal{F}}_t\big](t,\omega)\Big\}.\nn
	\end{align}
	
	Given $\psi$, we can now define $M^\psi$ by the Doob-Meyer decomposition of $G^\psi$, that is
	\begin{align}\label{eqn:Doob-Meyer}
		G_t^\psi=M_t^\psi-A_t^\psi,
	\end{align}
	where $A_t^\psi$ is an increasing process with $A_0^\psi=0$. This theory has been developed in \cite{mertens1972theorie}, for which we refer also to \cite{bismut1979potential}.  Unfortunately, the theory is stated with slightly different assumptions, so we reproduce the results we need.

	\begin{lemma}\label{lem:value_process}
		We suppose {\normalfont \ref{itm:lower_semicontinuity}} and {\normalfont\ref{itm:bounded_stopping}}. The dual problem has the equivalent expression:
		\begin{align}\label{eqn:concave_maximization}
			D_S(\mu,\nu) = \sup_{\psi\in LSC_b(O_\mathfrak{C})}\Big\{ \int_{O_\mathfrak{C}}\psi(y)\nu(dy)-\mathbb{E}^{\mathbb{P}^\mu}\big[G_0^\psi\big]\Big\}.
		\end{align}
		In particular, for every $\psi\in C_b(O_\mathfrak{C})$, we have that $(\psi,M^\psi)\in \mathcal{A}_S$, and $\mathbb{P}^\mu$ almost surely  $G_\sigma^\psi \leq M_\sigma$ for all $(\psi,M)\in \mathcal{A}_S$ and $\sigma \in \mathcal{S}$.

		Moreover, $G_\sigma^\psi=\psi(\mathfrak{C})-S_{\tau_\mathfrak{C}}$ holds $\mathbb{P}^\mu$ almost surely whenever $\sigma\geq \tau_\mathfrak{C}$.
	\end{lemma}

	\begin{proof}
	We fix $\psi\in LSC_b({O}_\mathfrak{C})$ and consider an increasing sequence $\psi^i\in C_b({O}_\mathfrak{C})$ that converges pointwise to $\psi$. For each $\psi^i$, the process $(\psi^i(X_t)-S_t)_{t\in \R^+}$ satisfies the assumptions of \cite{bismut1979potential} in view of \ref{itm:lower_semicontinuity}, making $G^{\psi^i}$ a regular supermartingale, i.e.\
	$$
		\E^{\mathbb{P}^\mu} \big[G^{\psi^i}_{\sigma^k}\big]\rightarrow\E^{\mathbb{P}^\mu} \big[G^{\psi^i}_{\sigma}\big]
	$$
	whenever $\sigma^k\rightarrow \sigma$ from the right.  We clearly have that $G^{\psi}\geq G^{\psi^i}$. For $\sigma\in \mathcal{S}$, there is $\tau\geq \sigma$ that attains the value of $G^\psi$, such that
	$$
		\mathbb{E}^{\mathbb{P}^\mu}\big[G^\psi_\sigma\big] = \mathbb{E}^{\mathbb{P}^\mu}\big[\psi(X_\tau)-S_\tau\big],
	$$
	and
	$$
		\E^{\mathbb{P}^\mu}\big[G^\psi_\sigma-G^{\psi^i}_\sigma\big]\leq \E^{\mathbb{P}^\mu}\big[\psi(X_\tau)-\psi^i(X_\tau)\big],
	$$
	which converges to zero by the dominated convergence theorem.  In particular, we have
	$$
		\inf_{x\in O_{\mathfrak{C}}}\psi(x)-\mathbb{E}^{\mathbb{P}^\mu}\big[S_\sigma\big]\leq \mathbb{E}^{\mathbb{P}^\mu}\big[G_\sigma^\psi\big]\leq \sup_{x\in O_\mathfrak{C}} \psi(x)-\inf_{t,\omega} S_t(\omega),
	$$
	so $G^\psi$ is uniformly integrable.  The supermartingale property for $G^\psi$ follows simply from noting that $G_t^\psi \geq \E^{\mathbb{P}^\mu}[G_\sigma^\psi|\mathcal{F}_t]$ for any $\sigma\in \mathcal{S}_t$. By setting $\sigma = t$  we see that  $G_t^\psi\geq \psi(X_t)-S_t$. 

	We take $M_t^\psi$ to be defined as the unique martingale of the Doob-Meyer decomposition \eqref{eqn:Doob-Meyer} and $M_t^\psi\geq G_t^\psi$, and it follows  $(\psi,M^\psi)\in \mathcal{A}_S$ with $M^\psi_0=G^\psi_0$ when $\psi\in C_b(O_\mathfrak{C})$.  Finally, we may find $\hat{\psi}\in C_b(O_\mathfrak{C})$ with $\hat{\psi}\leq \psi$ and arbitrarily close cost by lower-semicontinuity of ${\psi}$, which implies the inequality $\geq$ of (\ref{eqn:concave_maximization}).

	For each  $M$ with $(\psi, M) \in \mathcal{A}_S$ and for all $\bar{\mathbb{P}}\in \mathcal{T}_t(\mu)$, we have 
	$$
		M_t(\omega) = \mathbb{E}^{\bar{\mathbb{P}}}\big[M_T|\bar{\mathcal{F}}_t\big](t,\omega)\geq \mathbb{E}^{\bar{\mathbb{P}}}\big[\psi(X_T) - S_T\big|\bar{\mathcal{F}}_t\big](t,\omega),
	$$
	thus $M_t\geq G_t^\psi$, which completes the proof of \eqref{eqn:concave_maximization}. 

	That $G_t^\psi=\psi(\mathfrak{C})-S_{\tau_\mathfrak{C}}$ holds $\mathbb{P}^\mu$ almost surely whenever $t\geq \tau_\mathfrak{C}$ follows directly from (\ref{eqn:value_process}), \ref{itm:bounded_stopping} and the assumptions on $X$.
 	\end{proof}

 	We now state the duality that is central to our analysis, which slightly extends  the duality of \cite{beiglboeck2017optimal} in some ways although is simplified by our assumption of bounds from the killing time \ref{itm:bounded_stopping} (see also a similar proof in a more specific setting in \cite{GKPS}). 

	\begin{theorem}\label{thm:weak_duality}
		We suppose {\normalfont\ref{itm:lower_semicontinuity}} and {\normalfont\ref{itm:bounded_stopping}}.
		Then, including the possible value of $+\infty$,
		$$
			\mathcal{D}_S(\mu,\nu) = \mathcal{P}_S(\mu,\nu),
		$$
		and if $\mathcal{P}_S(\mu,\nu)<+\infty$, there is $\bar{\mathbb{P}}^*\in \mathcal{T}(\mu,\nu)$ such that $\mathcal{P}_S(\mu,\nu) = \mathbb{E}^{\bar{\mathbb{P}}^*}[S_{T}]$.
	\end{theorem}
	\begin{proof}
	The proof is a standard application of convex duality, which we sketch for completeness.	
We let $V(\nu)=\mathcal{P}_S(\mu,\nu)$, and we have that $V(\nu)$ is convex since if $\bar{\mathbb{P}}^0\in \mathcal{T}(\mu,\nu^0)$ and $\bar{\mathbb{P}}^1\in \mathcal{T}(\mu,\nu^1)$ then a $(1-\lambda)\bar{\mathbb{P}}^0+\lambda \bar{\mathbb{P}}^1\in \mathcal{T}(\mu,(1-\lambda)\nu^0+\lambda \nu^1)$, and the cost is linear in $\bar{\mathbb{P}}$.  Furthermore, $\nu\mapsto V(\nu)$ is lower-semicontinuous because the randomized stopping times with $T\leq \tau_\mathfrak{C}$ is compact, and the cost is lower-semicontinuous by \ref{itm:lower_semicontinuity}. We can express the Legendre transform as
   \begin{align*}
   	V^{*}(\psi)=&\ \sup_{\nu \in \mathcal{M}(O_\mathfrak{C})}\Big\{\int_{O_\mathfrak{C}} \psi(y) \nu(dy) - V(\nu)\Big\}\\
   	=&\ \sup_{\bar{\mathbb{P}}\in \mathcal{T}(\mu)}\Big\{\mathbb{E}^{\bar{\mathbb{P}}} \Big[\psi(X_T)   -S_T\Big] 
  \Big\},
  \end{align*}
and
\begin{align*}
	V^{**}(\nu)=&\ \sup_{\psi\in C_b(O_\mathfrak{C})}\Big\{\int_{O_\mathfrak{C}}\psi(y)\nu(dy)-\sup_{\bar{\mathbb{P}}\in \mathcal{T}(\mu)}\Big\{\mathbb{E}^{\bar{\mathbb{P}}} \Big[\psi(X_T)   -S_T\Big] 
  \Big\}\\
  =&\  \sup_{\psi\in C_b(O_\mathfrak{C})}\Big\{\int_{O_\mathfrak{C}}\psi(y)\nu(dy)-\mathbb{E}^{\mathbb{P}^\mu}\big[G^\psi_0\big]\Big\},
\end{align*}
where the second line follows from Lemma \ref{lem:value_process}.  Since $V^{**}(\nu)=V(\nu)$ by convexity and lower-semicontinuity, this completes the proof that  that $\mathcal{D}_S(\mu,\nu) = \mathcal{P}_S(\mu,\nu)$ as relaxing to $\psi\in LSC_b(O_\mathfrak{C})$ does not change the cost as in Lemma \ref{lem:value_process}.

When $\mathcal{P}_S(\mu,\nu)<+\infty$, by compactness of $\mathcal{T}(\mu,\nu)$ and \ref{itm:lower_semicontinuity}, we have the existence of a minimizer $\bar{\mathbb{P}}^*$.
	\end{proof}

We have the following `verification' type result for the dual optimizer. 
 	\begin{theorem}\label{thm:verification} 
 		Suppose {\normalfont \ref{itm:lower_semicontinuity}} and {\normalfont\ref{itm:bounded_stopping}} and that $\psi\in LSC_b(O_\mathfrak{C})$ attains the maximum of $\mathcal{D}_S(\mu,\nu)$, and $\bar{\mathbb{P}}^*\in \mathcal{T}(\mu,\nu)$ minimizes {\normalfont (\ref{eqn:primal})}.  Then $\bar{\mathbb{P}}^*$ maximizes
		\begin{align}\label{eqn:auxilliary}
			\E^{\bar{\mathbb{P}}}\big[\psi(X_T)-S_T\big]
		\end{align}
		over $\bar{\mathbb{P}}\in \mathcal{T}(\mu)$.

		Furthermore, for any maximizer $\bar{\mathbb{P}}\in \mathcal{T}(\mu)$ of {\normalfont (\ref{eqn:auxilliary})}, we have
\begin{enumerate}
  \item $G^\psi_{T} = \psi(X_{T})-S_{T}$ holds $\bar{\mathbb{P}}$ almost surely, 
  \item $G^\psi_{t\wedge T}$ is a $(\bar{\Omega},\bar{\mathbb{F}},\bar{\mathbb{P}})$ martingale, i.e., $M^\psi_{t\wedge T} = G^{\psi}_{t\wedge T}$ holds $\bar{\mathbb{P}}$ almost surely for all $t\in \mathbb{R}^+$.
\end{enumerate}
		 	\end{theorem}
 	\begin{proof}
 		We note that
 		\begin{align*}
 			\sup_{\bar{\mathbb{P}}\in \mathcal{T}(\mu)}\Big\{\E^{\bar{\mathbb{P}}}\big[\psi(X_T)-S_T\big]\Big\}=\mathbb{E}^{\mathbb{P}^\mu}\big[G_0^\psi\big]=\E^{\bar{\mathbb{P}}^*}\big[\psi(X_T)-S_T\big]
 		\end{align*}
 		by the duality of Theorem \ref{thm:weak_duality}, thus $\bar{\mathbb{P}}^*$ is a maximizer of (\ref{eqn:auxilliary}).

 		For any maximizer $\bar{\mathbb{P}}\in \mathcal{T}(\mu)$ of (\ref{eqn:auxilliary}), by definition of $G^\psi$, with $\bar{\mathbb{P}}$ probability 1, we have
 		$$
 			G^\psi_{T}\ge \psi(X_{T})-S_{T},
 		$$
 		and by the supermartingale property of $G^\psi$ we have
 		$$
 			\mathbb{E}^{\bar{\mathbb{P}}}[G^\psi_{T}]\le \mathbb{E}^{\mathbb{P}^\mu}[G^\psi_0] = \mathbb{E}^{\bar{\mathbb{P}}}\big[\psi(X_{T})-S_{T}\big]
 		$$
 		where the last equality is due to the fact that $\bar{\mathbb{P}}$ the maximum. Thus we have $G^\psi_{T} = \psi(X_{T})-S_{T}$ holds $\bar{\mathbb{P}}$ almost surely since $\ge$ holds almost surely and $\le$ holds in expectation. This proves (1).  It also implies that $\mathbb{E}^{\bar{\mathbb{P}}}[G^\psi_{T}]= \mathbb{E}^{\mathbb{P}^\mu}[G^\psi_0]$. Then, for $t\in \mathbb{R}^+$, from the supermartingale property we have 
 		$$
 			\mathbb{E}^{\bar{\mathbb{P}}}[G^\psi_{t \wedge T}]\geq \mathbb{E}^{\bar{\mathbb{P}}}[G^\psi_{T}]=\mathbb{E}^{\bar{\mathbb{P}}}[G^\psi_{0}]  		$$
 		yielding  $\mathbb{E}^{\bar{\mathbb{P}}}[G^\psi_{t \wedge T}] = \mathbb{E}^{\bar{\mathbb{P}}}[G^\psi_{0}]$. Since $t\mapsto G^\psi_{t\wedge T}$ is a supermartingale, this last property implies that it is a martingale. Since $M^\psi_T \geq G^\psi_T$ it immediately follows that they are equal $\bar{\mathbb{P}}$ almost surely, proving (2).
 	\end{proof}

The final lemma in this section selects a maximal $\psi^{\text{\it max}}$ given $\psi$, which also does not decrease the value.  This will play an important  role later for attainment of the problem \eqref{eqn:concave_maximization}. At this point $\psi^{\text{\it max}}$ is not necessarily bounded above, but when we apply the Lemma we will have a natural upper bound of $0$.
 	\begin{lemma}\label{lem:psi_maximization} 
 		We suppose {\normalfont\ref{itm:lower_semicontinuity}}, {\normalfont\ref{itm:bounded_stopping}} and $\psi\in LSC_b(O_\mathfrak{C})$. We let 
 		$$
 			\psi^{\text{\it max}}(y):=\sup_{\phi\in C_b({O}_\mathfrak{C})}\big\{\phi(y);\ \phi(X_\sigma(\omega))\leq G_\sigma^\psi(\omega) +S_\sigma(\omega),\ \forall\ \sigma\in \mathcal{S},\ \mathbb{P}^\mu-a.e.\ \omega\big\}.
 		$$  
 		Then we have the following:
 		\begin{enumerate}[label=\roman*.]
 			\item\label{itm:psi_increases} $\psi^{\text{\it max}}(y)\geq \psi(y)$\ for all $y\in O_{\mathfrak{C}}$;
 			\item\label{itm:same_value_process} $G^{\psi^{\text{\it max}}}_\sigma= G^\psi_\sigma,\ \forall\ \sigma\in \mathcal{S},\ \mathbb{P}^\mu-a.e.\ \omega$.
 		\end{enumerate}
 	\end{lemma}
 	\begin{proof}
 		We immediately note that $\psi^{\text{\it max}}$ is bounded below and lower-semicontinuous as the supremum of continuous functions.   Furthermore, $\psi$ can be expressed as the supremum of continuous functions which satisfy $\phi(X_\sigma)\leq G_\sigma^\psi + S_\sigma$ thus $\psi^{\text{\it max}}\geq \phi$ and \emph{\ref{itm:psi_increases}}\ follows.

The inequality $G^{\psi^{\text{\it max}}}\geq G^\psi$ is obvious from \emph{\ref{itm:psi_increases}}.
 		The pointwise inequality $\psi^{\text{\it max}}(X_\sigma)\leq G^\psi_\sigma+S_\sigma$ is maintained in the limit so in particular, thus
 		\begin{align*}
 			G^{{\psi}^{\text{\it max}}}_t=&\ \sup_{\sigma \in \mathcal{S}_t}\E^{\mathbb{P}^\mu}\big[\psi^{\text{\it max}}(X_\sigma)-S_\sigma\big|\mathcal{F}_t\big]\\
 			\leq&\ \sup_{\sigma \in \mathcal{S}_t}\E^{\mathbb{P}^\mu}\big[G^\psi_\sigma\big|\mathcal{F}_t\big] \leq G^\psi_t,
 		\end{align*}
 		using the supermartingale property of $G^\psi$, so \emph{\ref{itm:same_value_process}} follows.
 	\end{proof}

 	 	 \section{Pointwise Bounds}\label{sec:pointwise_bounds}
	We establish in this section pointwise bounds on the dual functions, which will be crucial in dual attainment in the later sections.
 	 	 We first introduce some additional structure to the processes.
 		 We recall that a stationary Feller process is given by a probability transition semigroup, such that for each $t>0$ the distribution of $X_t$ given $X_0=x$ is given by $P(t,x,\cdot)$, satisfying for $0<s<t$,
 			$$
 				P(t,x,\cdot)=\int_{O_\mathfrak{C}}P(t-s,y,\cdot)P(s,x,dy),
 			$$
 			and that $\lim_{t\rightarrow 0} \int_{O_\mathfrak{C}}f(y)P(t,\cdot,dy)=f$ uniformly for all $f\in C_b(O_\mathfrak{C})$.

 			We note that the processes beginning at $X_t=x$, are independent processes in a fixed probability space $(\Omega,\mathbb{F},\mathbb{P}^{x})$.  We let $\mathbb{E}^{x}$ denote expectation with respect to this probability space. We let $\mathcal{S}^x$ denote the (nonrandomized) stopping times given in  this probability space.  
		For additional references on optimal stopping in this setting see \cite{lamberton1998american} chapter 2 and references therein, as well as \cite{el1992probabilistic}.

 	 We define $\psi^{\text{\it r\'{e}.}}$ to be the r\'{e}duite of $\psi$.  This function corresponds to the superharmonic envelope when the process is Brownian motion. For $\psi\in LSC_b(O_\mathfrak{C})$,
 		\begin{align}\label{eqn:superharmonic}
 			\psi^{\text{\it r\'{e}.}}(x):=&\ \sup_{\sigma\in \mathcal{S}^x}\mathbb{E}^x\big[\psi(X_\sigma)\big].
 		\end{align}
 		
 		We say that \emph{balayage} holds, or 
 		$
 			\mu\prec\nu
 		$
 		if
 		$$
 			\int_{O_\mathfrak{C}}\psi(y)\nu(dy)\leq \int_{O_\mathfrak{C}}\psi(x)\mu(dx)
 		$$
 		for all supermedian functions, i.e.\ whenever $\psi=\psi^{\text{\it r\'{e}.}}$.

We first show a uniform pointwise bound assuming the following additional assumptions.
 		\begin{enumerate} [label=\textbf{(B\arabic*)}] \setcounter{enumi}{-1}
 			\item \label{itm:Feller}$X$ is a stationary Feller process.  Furthermore, we suppose that for all $\psi\in C_b(O_\mathfrak{C})$ the r\'{e}duite function of (\ref{eqn:superharmonic}) is continuous and bounded,  $\psi^{\text{\it r\'{e}.}}\in C_b(O_\mathfrak{C})$.  (The second assumption holds for all Feller processes if $O_\mathfrak{C}$ is compact; see \cite{el1992probabilistic}.)
 			\item \label{itm:submartingale} We suppose that $S_0=0$ and $S$ is a $(\Omega,\mathbb{F},\mathbb{P}^\mu)$-submartingale.
 	\end{enumerate}

 		The following corollary of Theorem \ref{thm:weak_duality} recovers a result of Rost \cite{rost1971stopping}.
	\begin{corollary}\label{cor:Strassen}
		Given {\normalfont \ref{itm:Feller}}, there exists $\bar{\mathbb{P}}\in \mathcal{T}(\mu,\nu)$ if and only if $\mu\prec \nu$.
	\end{corollary}
	\begin{proof}
		We fix $S=0$, in which case $G^\psi_t=\psi^{\text{\it r\'{e}.}}(X_t)$.  If $\mu\prec \nu$ then for any $\psi\in LSC_b(O_\mathfrak{C})$, by Lemma \ref{lem:value_process} and the definition of $\psi^{\text{\it r\'{e}.}}$ and the balayage we have
		\begin{align*}
			\int_O \psi(y)\nu(dy)-\E^{\mathbb{P}^\mu}\big[G_0\big]\leq&\ \int_O \psi^{\text{\it r\'{e}.}}(y)\nu(dy)-\int_O \psi^{\text{\it r\'{e}.}}(x)\mu(dx)\leq 0.
		\end{align*}
		It follows from Theorem \ref{thm:weak_duality} that $\mathcal{P}_S(\mu,\nu)=0$ and there exists $\bar{\mathbb{P}}\in \mathcal{T}(\mu,\nu)$.

		If $\mu\not\prec \nu$, then there exists a supermedian function $\phi$,  which satisfies
		$$
			\int_O \phi(y)\nu(dy)-\int_O \phi(x)\mu(dx)> 0,
		$$ 
		in which case $G^{\lambda\phi}_t=\lambda\phi(X_t)$ for any $\lambda>0$. Taking $\lambda$ to $+\infty$ we see that
		$$
			\mathcal{D}_S(\mu,\nu)=+\infty
		$$
		and from Theorem  \ref{thm:weak_duality} we have that $\mathcal{T}(\mu,\nu)$ is empty.
	\end{proof}

 	The next lemma in this section verifies a mean value type property for the r\'{e}duite $\psi^{\text{\it r\'{e}.}}$, which asserts that $\psi^{\text{\it r\'{e}.}}(X_t)$ is a martingale up until the set where it touches the obstacle $\psi$, which is an extension of Theorem \ref{thm:verification} in the case that $S=0$. 
	
	\begin{lemma}\label{lem:superharmonic_snell} 
	Assume {\normalfont\ref{itm:Feller}} and that $\psi\in C_b({O}_\mathfrak{C})$. Then the first hitting time,
		\begin{align*}
			\eta:= \inf\{t;\ \psi\big(X_t\big)=\psi^{\text{\it r\'{e}.}}\big(X_t\big)\},
		\end{align*}
		 attains the supremum of {\normalfont(\ref{eqn:superharmonic})} with 	
\begin{align}\label{eqn:SH-eta}
 \hbox{$\E^x[\psi(X_{\eta})]=\psi^{\text{\it r\'{e}.}}(x)$ and 
$ \psi (X_\eta) = \psi^{\text{\it r\'{e}.}}(X_\eta).$}
\end{align}	
		Moreover, for any  randomized stopping time $\bar{\mathbb{P}}\in \mathcal{T}(\delta_x)$, we have the mean value property
		\begin{align}\label{eqn:equal-in-harmonic}
			\mathbb{E}^{\bar{\mathbb{P}}}\big[\psi^{\text{\it r\'{e}.}}(X_{T\wedge\eta})\big] = \psi^{\text{\it r\'{e}.}}(x).
		\end{align}
\end{lemma}

	\begin{proof}
	The proof is standard, but we give it here for completeness.
		First,  by continuity of $\psi$ and $\psi^{\text{\it r\'{e}.}}$ from \ref{itm:Feller}, and right-continuity of the paths of $X_t$, we have that $\psi^{\text{\it r\'{e}.}}(X_{\eta})=\psi(X_{\eta})$.  
		 To check that  $\eta$ attains the supremum of \eqref{eqn:superharmonic}, notice that if  $\bar{\mathbb{P}}\in \mathcal{T}(\delta_x)$ is an optimal randomized stopping time for \eqref{eqn:superharmonic}, then, $\bar{\mathbb{P}}$ almost surely, $\psi(X_{T})=\psi^{\text{\it r\'{e}.}}(X_{T})$, by the dynamic programming principle as in the proof of Theorem \ref{thm:verification}, hence $T\geq \eta$. Therefore,  the supermedian property implies $\mathbb{E}^{\bar{\mathbb{P}}}[\psi^{\text{\it r\'{e}.}}(X_{T})]\leq \mathbb{E}^x[\psi^{\text{\it r\'{e}.}}(X_{\eta})]$, showing $\eta$ attains the supremum of \eqref{eqn:superharmonic}, namely, $\psi^{\text{\it r\'{e}.}}(x)=\mathbb{E}^x[\psi^{\text{\it r\'{e}.}}(X_{\eta})]$.  Using the supermedian property again we get for any randomized stopping time $\bar{\mathbb{P}}\in \mathcal{T}(\delta_x)$,
		 $$
			\psi^{\text{\it r\'{e}.}} (x) \ge \mathbb{E}^{\bar{\mathbb{P}}} \big[\psi^{\text{\it r\'{e}.}}(X_{T\wedge\eta})\big]\geq	\mathbb{E}^x\big[\psi^{\text{\it r\'{e}.}}(X_{\eta})\big]=\psi^{\text{\it r\'{e}.}}(x),
		$$
		proving \eqref{eqn:SH-eta} and \eqref{eqn:equal-in-harmonic}.
\end{proof}

	We now normalize the value process $G^\psi$ by the r\'{e}duite of $\psi$.
	The following provides a  key ingredient in our dual attainment argument, which generalizes Proposition 4.6 of \cite{GKP-Monge} with essentially the same proof in this more general setting.  
 	\begin{proposition}\label{prop:psi_normalization}
 		Suppose {\normalfont \ref{itm:lower_semicontinuity}, \ref{itm:bounded_stopping}, \ref{itm:Feller}, \ref{itm:submartingale}}, and $\psi\in C_b({O}_\mathfrak{C})$.  If we let $\bar{\psi}=\psi-\psi^{\text{\it r\'{e}.}}$ then we have 
 		$$
 			G_t^{\bar{\psi}} = G_t^{\psi}-\psi^{\text{\it r\'{e}.}}(X_t).
 		$$
 	\end{proposition}
 	\begin{proof}
 		First, we show $\geq$, which is an easy step and does not require $S$ to be a submartingale.   We let $\overline{\mathbb{P}}\in \mathcal{T}_t(\mu)$ attain the supremum of the definition of $G_t^\psi$, (\ref{eqn:value_process}), i.e.\
 		$$
 			G_t^\psi = \mathbb{E}^{\overline{\mathbb{P}}}\big[\psi(X_{T})-S_{T} \big| \mathcal{F}_t\big],
 		$$
 		thus by the supermedian property of $\psi^{\text{\it r\'{e}.}}$ and the definition of $G_t^{\bar{\psi}}$,
 		$$
 			G_t^\psi-\psi^{\text{\it r\'{e}.}}(X_t) \leq \mathbb{E}^{\overline{\mathbb{P}}}\big[\psi(X_{T})-S_{T} -\psi^{\text{\it r\'{e}.}}(X_{T}) \big| \mathcal{F}_t\big]\leq G_t^{\bar{\psi}}.
 		$$

 		For the other direction we let $\eta:= \inf\{s\geq t;\ \psi(X_s)=\psi^{\text{\it r\'{e}.}}(X_s)\}$ as in Lemma \ref{lem:superharmonic_snell}.  We have that $\eta$ is a $\mathcal{F}_t$ stopping time because $X$ is adapted to $\mathcal{F}_t$. We let $\bar{\mathbb{P}}\in \mathcal{T}_t(\mu)$ attain the supremum of the definition of $G_t^{\bar{\psi}}$, then:
 		\begin{itemize}
 		\item by the definition of $G^{\bar{\psi}}$ and using from Lemma \ref{lem:superharmonic_snell} that $\psi^{\text{\it r\'{e}.}}(X_t) = \mathbb{E}^{\bar{\mathbb{P}}}\big[\psi^{\text{\it r\'{e}.}}(X_{T\wedge\eta})\big]$, 
 		$$\mathbb{E}^{\bar{\mathbb{P}}}\big[\bar{\psi}(X_{T\wedge \eta}) - S_{T\wedge \eta}\big| \mathcal{F}_t\big] =\mathbb{E}^{\bar{\mathbb{P}}}\big[\psi(X_{T\wedge \eta}) - S_{T\wedge \eta}\big| \mathcal{F}_t\big] - \psi^{\text{\it r\'{e}.}}(X_t) \leq G_t^{\psi}-\psi^{\text{\it r\'{e}.}}(X_t);$$
 		\item also from Lemma \ref{lem:superharmonic_snell} we have $\mathbb{E}^{\bar{\mathbb{P}}}\big[{\psi}(X_{\eta})\big]=\mathbb{E}^{\bar{\mathbb{P}}}\big[\psi^{\text{\it r\'{e}.}}(X_{\eta})\big]$ so 
 		$$
 			\mathbb{E}^{\bar{\mathbb{P}}}\big[\bar{\psi}(X_T)-\bar{\psi}(X_{T\wedge \eta})\big| \mathcal{F}_t\big]\leq 0;
 		$$
 		\item and by the submartingale property of $S$,
 		$$
 			-\mathbb{E}^{\bar{\mathbb{P}}}\big[S_T-S_{T\wedge \eta}\big| \mathcal{F}_t\big]\leq 0.
 		$$

 	\end{itemize}

 	It follows that
 	\begin{align}
 		G_t^{\bar{\psi}} = &\ \mathbb{E}^{\bar{\mathbb{P}}}\big[\bar{\psi}(X_T)-S_T \big| \mathcal{F}_t\big]\nn\\
 		=&\ \mathbb{E}^{\bar{\mathbb{P}}}\big[\bar{\psi}(X_{T\wedge \eta}) - S_{T\wedge \eta}\big| \mathcal{F}_t\big]\nn\\
 		&\ +\mathbb{E}^{\bar{\mathbb{P}}}\big[\bar{\psi}(X_T)-\bar{\psi}(X_{T\wedge \eta})\big| \mathcal{F}_t\big]\nn\\
 		&\ -\mathbb{E}^{\bar{\mathbb{P}}}\big[S_T-S_{T\wedge \eta}\big| \mathcal{F}_t\big]\nn\\
 		\leq&\ G_t^{\psi}-\psi^{\text{\it r\'{e}.}}(X_t),\nn
 	\end{align}
 	completing the proof.
 	\end{proof}

\section{Dual Attainment in the Discrete Setting}

We illustrate the results in the previous section in the simple case of discrete Markov process.
 	\begin{theorem}\label{thm:finite_attainment}
 		Suppose that $O$ is a discrete set { (i.e. with at most countable elements)}
		 and that $S$ satisfies {\normalfont\ref{itm:lower_semicontinuity}, \ref{itm:bounded_stopping}, \ref{itm:submartingale}}, and {\normalfont\ref{itm:Feller}} holds (i.e. $X$ is Markov).  Furthermore, we assume that $S_{\tau_\mathfrak{C}}$ is uniformly bounded, i.e.\ $S_{\tau_\mathfrak{C}}(\omega)\leq K$ for $\mathbb{P}^\mu$-a.e.\ $\omega$. Then the dual problem is attained at  $\psi^*\in C_b(O_\mathfrak{C})$  with $-K \le \psi^*\leq 0$ and $\psi^*(\mathfrak{C})=0$.  
 	\end{theorem}
 	\begin{proof}
 	Given $\psi\in C_b(O_\mathfrak{C})$ we let $\bar{\psi}=\psi-\psi^{\text{\it r\'{e}.}}$ as in Proposition \ref{prop:psi_normalization}. 
 	From $\bar \psi \le 0$ and the submartingale property of $S$, we have
 			$$
 				G^{\bar{\psi}}_t = \sup_{\bar{\mathbb{P}}\in \mathcal{T}_t(\mu)} \E^{\bar{\mathbb{P}}}\Big[\bar{\psi}(X_T)-S_T\Big|\mathcal{F}_t\Big]  \le  -S_t . 
 			$$
			On the other hand, 		
\begin{align*}
 G^{\bar{\psi}}_t + S_t \geq \E^{\mathbb{P}^\mu}\Big[\bar{\psi}(X_{\eta})-S_{\eta} + S_t \Big|\mathcal{F}_t\Big] \geq - \E^{\mathbb{P}^\mu}\Big[S_\eta - S_t \Big| \mathcal{F}_t\Big]
\end{align*} 
where $\eta:=\inf\{s\geq t;\ \psi(X_s)=\psi^{\text{\it r\'{e}.}}(X_s)\}$ as in Lemma \ref{lem:superharmonic_snell}. Notice that  $$\E^{\mathbb{P}^\mu}\big[S_\eta - S_t \big| \mathcal{F}_t\big] \le K$$ from the boundedness assumption on $S$ with respect to stopping times.  These show that $-K \le G^{\bar\psi}+S \le 0$.

 			We now follow the maximization procedure of Lemma \ref{lem:psi_maximization} to obtain ${\bar{\psi}}^{\text{\it max}}$.  Then $-K\leq {\bar{\psi}}^{\text{\it max}}\leq 0$ follows from $-K \le G^{\bar\psi}+S \le 0$, and that fact that for any open neighborhood $Q$ of $y$, there is $\sigma \in \mathcal{S}$ such $X_\sigma\in Q$ has positive probability.
			Finally, 		Proposition \ref{prop:psi_normalization} combined with Lemma \ref{lem:psi_maximization}    implies that the dual value  for  $\psi^{\text{\it max}}$ is greater than or equal to the value  for $\psi$. 
 		We have restricted the optimization problem to the set of $\psi\in C_b(O_\mathfrak{C})$ where  $-K\leq {\psi}\leq 0$ and $\psi(\mathfrak{C})=0$.   

 		This subset of $C_b(O_\mathfrak{C})$ is compact since either $[-K, 0]^N  \subset \R^N$ or $[-K, 0]^\infty  \subset \R^\infty$ is compact.

 		The dual value is upper-semicontinuous as the infimum of continuous linear functionals, in particular
 		$$
 			-\mathbb{E}^{\mathbb{P}^\mu}\big[G_0^\psi\big]=\inf_{\bar{\mathbb{P}}\in \mathcal{T}(\mu)}\Big\{-\mathbb{E}^{\bar{\mathbb{P}}}\big[\psi(X_T)-S_T\big]\Big\},
 		$$
 		 and dual attainment follows as the maximization of an upper-semicontinuous function on a compact set.  
 	\end{proof}

 \section{Dual Attainment in Hilbert Space}\label{sec:dual_attainment}
This section gives attainment of the dual problem \eqref{eqn:concave_maximization}, which is equivalent to \eqref{eqn:dual}. The dual optimizer is found in a Hilbert space, which we define using a Dirichlet form as follows. 

  Let $O$ be equipped with a positive finite Borel measure $m$. (We ignore the cemetery state $\mathfrak{C}$ as we will assume from here on out  that all functions have value $0$ on $\mathfrak{C}$.)  Assume that there is a symmetric semi-definite (Dirichlet) form 
 \begin{align*}
 \mathcal{E}: L^2(O; m) \times L^2 (O; m) \to \R \cup \{+ \infty\}. 
\end{align*}
 We let 
\begin{align*}
  \mathcal{H}: = \{ u \in L^2 (O; m); \ \mathcal{E}(u, u) < +\infty\}.
\end{align*}
Note $\mathcal{H}$ is in general a Hilbert space with the inner product $\mathcal{E}(u,v) + \int_O u(x) v(x)\, m(dx)$, but we will consider when the Dirichlet form defines a Hilbert space without the additional $L^2$ product, i.e.\ the Poincar\'{e} inequality holds. 
We let $\mathcal{H}^*$ denote the dual space of linear functionals with respect to the $L^2$ inner product.  In particular, we will say a measure $\gamma$ belongs to $\mathcal{H}^*$, if there exists $U^\gamma\in \mathcal{H}$ such that
$$
	\int_O f(x)\gamma(dx)=\mathcal{E}(f,U^\gamma),\ \forall\ f\in C_b(O)\cap \mathcal{H},
$$
in which case $\|\gamma\|_{\mathcal{H}^*}=\|U^\gamma\|_{\mathcal{H}}$.
We  abuse the notation and let $\Delta$ denote the generator of the Dirichlet form, and $\mathcal{H}_0$ be the set of $f\in C_b(O) \cap \mathcal{H}$ such that $\Delta f\in C_b(O)\cap L^2(O;m)$.  In other words, for $f\in \mathcal{H}_0$ and $g\in \mathcal{H}$ we have
$$
	\mathcal{E}(f,g)=\int_O g(x)\big(-\Delta f(x)\big)m(dx).
$$
We suppose that $\Delta$ generates $X_t$ in the sense that for each $f\in \mathcal{H}_0$ and $\sigma\in \mathcal{S}^x$,
\begin{align}\label{eqn:generator}
f(x)=\mathbb{E}^x\Big[f(X_\sigma)-\int_0^\sigma \Delta f(X_t)dt\Big]. 
\end{align}

We say that $g\in LSC_b(O)$ is a supersolution to $\Delta g\leq h$ for $h\in LSC_b(O)$ in the viscosity sense if whenever $f\in \mathcal{H}_0$ touches $g$ from below at $x\in O$, i.e.\ $f(x)=g(x)$ and $f(y)\leq g(y)\ \forall\ y\in O$, then
$$
	\Delta f(x)\leq h(x).
$$
We say that $g\in \mathcal{H}$ is a supersolution to $\Delta g\leq h$ for $h\in L^2(O;m)$ in the weak sense if 
$$
\mathcal{E}(f,g) \geq -\int_O h(x)f(x)m(dx)
$$
for all  $f\in \mathcal{H}_0$ with $f\geq 0$.

We list the assumptions we need for our main results. 
\begin{enumerate}[label=\textbf{(C\arabic*)}]\setcounter{enumi}{-1}

\item \label{assumptionPoincare}\, {\bf [Poincar\'e inequality]} $\exists\ C_p >0$ such that  $\displaystyle \mathcal{E}(u, u) \ge C_p^{-1} \int_O |u(x)|^2 m(dx)$ for all  $u \in \mathcal{H}$.  In particular, we take the norm and inner product on $\mathcal{H}$ to be given solely by $\mathcal{E}$.

 \item\label{assumptionL} \, {\bf [Continuity/Variational Equivalence]} For $h\in LSC_b(O)\cap L^2(O,m)$, we have $\psi \in LSC_b(O)$ satisfies
 $ \Delta \psi \le h$  in the viscosity sense if and only if $\psi\in \mathcal{H}$ is bounded above and satisfies $\Delta \psi \le h$ in the weak sense.

\item\label{assumptionS} \, [{\bf Semi-supermartingale}] There is $D\geq 0$, such that the cost satisfies $\displaystyle \mathbb{E}^{\mathbb{P}^\mu}\big[ S_\sigma\big|\mathcal{F}_t \big] - S_t \le \mathbb{E}^{\mathbb{P}^\mu}\big[ D(\sigma-t)\big|\mathcal{F}_t\big]$ for all $\sigma\in \mathcal{S}_t$ and $\mathbb{P}^\mu$ almost surely.

\item \label{assumptionMeasures}\, [{\bf Balayage}] We have $\mu\in \mathcal{H}^*$ and $\mu\prec \nu$. 
\end{enumerate}

\begin{example}\label{ex:elliptic} \
  	\begin{enumerate}
	\item $O= \R^d$. 
	 		\item The state process, $X_t$, is a $d$-dimensional diffusion process generated by a smooth uniformly elliptic operator $\Delta = \sum_{i=1}^d\sum_{j=1}^d\partial_i a_{ij}\partial_j -\beta$, with killing rate $\beta>0$. 
		\item $\displaystyle \mathcal{E}(u, v) = \int_{\R^d}\Big( \sum_{i=1}^d\sum_{j=1}^d a_{ij}(x) \partial_i u(x) \,  \partial_j v(x) +\beta u(x)v(x)\Big)dx$, where $m$ is Lebesgue measure on $\R^d$ and $\mathcal{H}\equiv H^1(\R^d)$.

	 	\item  The cost process, $S_t:=\int_0^t L(t,X_t)dt$ where $L$ is continuous and
		$$
			0\leq L(t,x)\leq D. 
		$$
	\end{enumerate}
	The killing rate that enforces the Poincar\'{e} inequality also causes $\mathbb{E}^{\mathbb{P}^\mu}[\tau_\mathfrak{C}]<+\infty$. 
\end{example}

\begin{example}\label{ex:riemann}\ 
\begin{enumerate}
 \item $O$ a geodesically convex bounded domain in a non positively curved Riemannian manifold.
 \item  $X_t$ is the Riemannian Brownian motion
 \item $\Delta$ is the Laplace-Beltrami operator with Dirichlet boundary conditions on $\partial O$.
 \item $\mathcal{E}(u, v) = \int_O g\big( \nabla u(x) , \nabla v(x)) vol_g(dx) $ where $g$ is the Riemannian metric and $vol_g$ is the corresponding volume form. 
 \item  The cost process, $S_t=c(X_0,X_t)$ where
		$$
			0\leq \Delta_y c(x,y)\leq D
		$$
		for all $x,y\in O$.

\end{enumerate}
 
\end{example}
In both of these examples one can check the non-trivial \ref{assumptionL} by viscosity solution theory as in \cite{GKP-Monge}.

\begin{example}
In example {\normalfont \ref{ex:elliptic}}, the uniformly elliptic operator can be replaced with the fractional Laplacian, yielding fractional Brownian motion on $\R^d$.
\end{example}
In this section we prove the main result of the paper on the attainment of the dual problem $\mathcal{D}_S$.  We recall that $\mathcal{D}_S(\mu,\nu)$ is defined as a supremum over the class $\mathcal{A}_S\subset C_b({O}_\mathfrak{C})\times B(\bar{\Omega})$, however, by Lemma \ref{lem:value_process}, is equal to the supremum over $\psi\in LSC_b({O}_\mathfrak{C})$ of the concave functional
 	\begin{align}\label{eqn:psi_functional}
 		U(\psi):=\int_{{O}}\psi(y)\nu(dy)-\mathbb{E}^{\mathbb{P}^\mu}\big[G_0^\psi\big].
 	\end{align}

 	We first introduce a subset $\mathcal{B}_{D}\subset LSC_b({O})\cap \mathcal{H}$, which plays a key role in our method. We always extend these functions to $LSC_b(O_\mathfrak{C})$ by $0$ on $\mathfrak{C}$.
	\begin{definition}\label{def:BD}
 		We say that $\psi\in \mathcal{B}_{D}$, if the following properties hold:
		\begin{enumerate}
 			\item $\psi\in LSC_b({O})\cap \mathcal{H}$.   
 			\item $\psi(y)\leq 0$ for all $y\in {O}$.
 			\item $\Delta \psi(x)\le D$  in the weak sense.
					\end{enumerate}
	\end{definition}
	Note that with  assumption \ref{assumptionL}, 
	the last condition follows if  $\Delta \psi(x)\le D$ in the sense of viscosity. 
	 Notice  that 
\begin{align*}
 \hbox{$\mathcal{B}_{D}$ is compact in the weak topology of $\mathcal{H}$}
\end{align*}
	 because of the uniform bound given by \ref{assumptionPoincare},
	 $$
	 	\mathcal{E}(\psi,f)\leq D\int_O |f(x)|m(dx)\leq D\, \sqrt{m(O)}\|f\|_{L^2(O;m)}\leq D\, \sqrt{m(O)}C_p\|f\|_{\mathcal{H}},
	 $$ 
	 and the Banach-Alaoglu theorem. We now prove that  $U:\mathcal{B}_{D} \rightarrow \R$ is concave and upper-semicontinuous. 
	\begin{proposition}\label{prop:compact_semicontinuous}
		We suppose {\normalfont \ref{itm:lower_semicontinuity}, \ref{itm:bounded_stopping}, \ref{itm:Feller}, \ref{itm:submartingale}, \ref{assumptionPoincare}-\ref{assumptionMeasures}}. The map $\psi\mapsto U(\psi)$ is concave and upper-semicontinuous on $\mathcal{B}_{D}$ with the weak topology of $\mathcal{H}$.
	\end{proposition}
	\begin{proof}
		Concavity and upper-semicontinuity follow from the structure as the infimum over linear functionals.    Since $-\E^{\mathbb{P}^\mu}[G_0^\psi] = \inf_{\bar{\mathbb{P}}\in \mathcal{T}(\mu)} \E^{\bar{\mathbb{P}}}[-\psi(X_T)+S_T]$, it suffices to show that the map
		$$
			\psi\mapsto \E^{\bar{\mathbb{P}}}[\psi(X_T)]
		$$
		is a continuous linear functional on $\mathcal{B}_{D}$ for any $\bar{\mathbb{P}}\in \mathcal{T}(\mu)$.  This fact follows from the fact that $\mu\prec \rho\sim X_T$ and that $\mu\prec \rho$ implies $\|\rho\|_{\mathcal{H}^*}\leq \|\mu\|_{\mathcal{H}^*}$. Indeed, for $\phi\in C_b(O)\cap \mathcal{H}$, we have that $\phi^{\text{\it r\'{e}.}}\in C_b(O)\cap\mathcal{H}$ with 
		$$\|\phi^{\text{\it r\'{e}.}}\|_{\mathcal{H}}^2=  \mathcal{E}(\phi^{\text{\it r\'{e}.}} ,\phi^{\text{\it r\'{e}.}}) \le \mathcal{E}(\phi, \phi) \le \| \phi\|_{\mathcal{H}}^2 $$
		by \ref{assumptionPoincare} and \ref{assumptionL}, since \ref{assumptionL} implies that $\phi^{\text{\it r\'{e}.}}$ minimizes $\mathcal{E}(u,u)$ over functions $u\in \mathcal{H}$ with $u\geq \psi$.
		Thus, by considering $\phi\in C_b(O)\cap \mathcal{H}$ with $\|\phi\|_{\mathcal{H}}=1$, we have
		$$
			\int_{O}\phi(y)\rho(dy)\leq \int_{O}\phi^{\text{\it r\'{e}.}}(y)\rho(dy)\leq \int_{O}\phi^{\text{\it r\'{e}.}}(x)\mu(dx)\leq  \|\mu\|_{\mathcal{H}^*}.
		$$
		Upper-semicontinuity follows from $\mu\in \mathcal{H}^*$, cf.\ \ref{assumptionMeasures}.
		\end{proof}

\begin{proposition}\label{prop:dual_normalization}
 		We suppose {\normalfont \ref{itm:lower_semicontinuity}, \ref{itm:bounded_stopping}, \ref{itm:Feller}, \ref{itm:submartingale}, \ref{assumptionPoincare}-\ref{assumptionMeasures}}.  Given $\psi\in LSC_b(O_\mathfrak{C})$ with $\psi(y)\leq 0$ for all $y\in O$ and $\psi(\mathfrak{C})=0$, we consider $\psi^{\text{\it max}}$ as in  Lemma {\normalfont\ref{lem:psi_maximization}}.  Then in the sense of viscosity,
 		 \begin{align}\label{eqn:uniform_psi_bound}
 			 \Delta \psi^{\text{\it max}}(y)\le D,
 		\end{align}
 		and $\psi^{\text{\it max}}\in \mathcal{B}_D$.

 		Consequentially,  the dual problem $\mathcal{D}_S(\mu,\nu)$ is reduced to $\mathcal{B}_D$, that is,
 		$$
			\mathcal{D}_S(\mu,\nu) = \sup_{\psi\in \mathcal{B}_D}U(\psi).
		$$  for the functional $U(\psi)$ of \eqref{eqn:psi_functional}.
 	 	\end{proposition}
\begin{proof}
 From Proposition \ref{prop:psi_normalization}, we can always normalize $\psi$ to $\bar{\psi}\leq 0$ with $\bar{\psi}(\mathfrak{C})=0$, after which we normalize to $\psi^{\text{\it max}}$ as in Lemma \ref{lem:psi_maximization}.
 We note that, as in Theorem \ref{thm:finite_attainment}, we have $\psi^{\text{\it max}}(y)\leq 0$ for all $y\in O$. Suppose that $\phi\in \mathcal{H}_0$ touches $\psi^{\text{\it max}}$ from below at $x$.  Then for any $\epsilon>0$ and $\delta>0$  there is $\sigma \in \mathcal{S}$ and a set $A\subset \mathcal{F}_\sigma$ with nonzero probability and $d(X_\sigma,x)<\delta$, such that 
$$
	\phi\big(X_\sigma(\omega)\big)+\epsilon\geq G_\sigma^\psi(\omega)+S_\sigma(\omega)
$$
for $\mathbb{P}^\mu$-a.e.\ $\omega\in A$. Then we have for all $s>0$, using \ref{assumptionS}  that for $\mathbb{P}^\mu$-a.e.\ $\omega \in A$ 
\begin{align*}
	\mathbb{E}^{\mathbb{P}^\mu}\big[\phi(X_{{\sigma}+s})\big|\mathcal{F}_\sigma\big]\leq&\ \mathbb{E}^{\mathbb{P}^\mu}\big[G_{{\sigma}+s}^\psi+S_{{\sigma}+s}\big|\mathcal{F}_\sigma\big]\\
	\leq& \ G^\psi_\sigma + S_\sigma + D s\\
	\leq&\ \phi\big(X_\sigma\big)+\epsilon+Ds.
\end{align*}
Because $X$ is a stationary Feller process with generator $\Delta$,
$$
	\mathbb{E}^{\mathbb{P}^\mu}\Big[\int_\sigma^{\sigma+s}\Delta\phi(X_r)dr\Big|\mathcal{F}_\sigma\Big]\leq Ds + \epsilon
$$ 
for all $s>0$ and $\mathbb{P}^\mu$ a.e.\ $\omega\in A$.  Let $\epsilon, \delta \to 0$ then continuity of $\Delta \phi$ implies that $\Delta \phi(x)\leq D$, and $\Delta \psi^{\text{\it max}}(x)\leq D$ in the sense of viscosity.  By \ref{assumptionL} we have that $\psi^{\text{\it max}}\in \mathcal{B}_D$, and the dual value has not decreased. This completes the proof. 
\end{proof}

We now state our main theorem on attainment of the dual problem, which follows immediately from the two preceding propositions.  
	\begin{theorem}\label{thm:strong_dual} 
		We assume {\normalfont \ref{itm:lower_semicontinuity}, \ref{itm:bounded_stopping}, \ref{itm:Feller}, \ref{itm:submartingale}, \ref{assumptionPoincare}-\ref{assumptionMeasures}}.  Then there is $\psi^*\in \mathcal{B}_D$ that is a maximizer of $\mathcal{D}_S(\mu, \nu)$, that is,
		$$
			\int_{{O}}\psi^*(y)\nu(dy) -\E^{\mathbb{P}^\mu}\big[G^{\psi^*}_0\big]=\mathcal{D}_S(\mu,\nu).
		$$
	\end{theorem}
	\begin{proof}
		By Proposition \ref{prop:dual_normalization} we may restrict to a maximizing sequence $\psi^i\in \mathcal{B}_D$ with $\psi^i = {\psi^i}^{max}$.  As noted above, $\mathcal{B}_D$ is compact in the weak topology of $\mathcal{H}$, and the result of Proposition \ref{prop:compact_semicontinuous} implies that for a subsequence $\psi^{i^k}\rightharpoonup \psi^\infty\in \mathcal{B}_D$ and
		$$
			U(\psi^\infty)\geq \lim_{k\rightarrow \infty}U({\psi}^{i_k})=\mathcal{D}_S(\mu,\nu)
		$$
		completing the proof.
	\end{proof}

\subsection{When the cost is the expected stopping time}\label{sec:expected_time}

 	The case that $S_t=t\wedge \tau_{\mathfrak{C}}$ is critical for understanding this problem. For a thorough exposition of related results for Brownian motion beginning at the origin in 1D see \cite{monroe1972embedding}. 
 	\begin{proposition}\label{prop:expected_time_transient}
 		We assume {\normalfont \ref{itm:lower_semicontinuity}, \ref{itm:Feller}, \ref{assumptionPoincare}-\ref{assumptionMeasures}}.  There is a unique function  (up to an additive constant) $h\in \mathcal{H}_0$ with $\Delta h=1$ in the weak sense on $O$ such that for any $\bar{\mathbb{P}}\in \mathcal{T}(\mu,\nu)$ we have
 		$$
 			\mathbb{E}^{\bar{\mathbb{P}}}\big[T\big]=\int_{O}h(x)\nu(dx)-\int_Oh(x)\mu(dx).
 		$$
 		In particular, any such $\bar{\mathbb{P}}$ is optimal for cost $S_t=t\wedge \tau_{\mathfrak{C}}$, and the dual problem is solved by $\psi=h$ and $M_t=G^h_t=h(X_t)-S_t$.
 	\end{proposition}
 	\begin{proof}
 		Existence and uniqueness (up to additive constant) of such $h$ in $\mathcal{H}$  with $\Delta h=1$, is immediate by  the property of the generator $\Delta$ of a stationary Feller process.  
		From $\Delta h =1$, clearly, $\Delta h\in C_b(O)\cap L^2(O,m)$ and that $h\in C_b(O)$ follows from \ref{assumptionL}.  We then calculate simply that
 		\begin{align*}
 			\mathbb{E}^{\bar{\mathbb{P}}}\big[T\big]= \mathbb{E}^{\bar{\mathbb{P}}}\Big[\int_0^T\Delta h(X_t)dt\Big]=\int_{O}h(x)\nu(dx)-\int_Oh(x)\mu(dx).
 		\end{align*}
 		Taking $\psi=h$, we find that $G^h_t=h(X_t)-S_t$, and since
 		$$
 			 \int_O h(x)\nu(dx)-\mathbb{E}^{\mathbb{P}^\mu}\big[G^h_0\big]=\int_{O}h(x)\nu(dx)-\int_Oh(x)\mu(dx)=\mathbb{E}^{\bar{\mathbb{P}}}\big[T\big],
 		$$
 		 $\psi=h$ is optimal for the dual problem.
 	\end{proof}

\section{The General Twist Condition}\label{sec:generalTwist}

We suppose now that the pair $(A,X)$ is a stationary Feller process with generator $\Delta_{a,x}$, and the cost decomposes as $S_t=\Lambda(A_t,X_t)$.  We assume that $A$ takes values in $\R^d$ and $a\mapsto \Lambda(a,x)$ is  differentiable with $(a,x)\mapsto \nabla_a\Lambda(a,x)$ continuous.

We interpret $A$ as an auxiliary parameter, which will provide structure for the optimal solutions.  In the examples below $A$ might be the time $A_t=t$, the initial position $A_t=X_0$, or a stochastic process coupled with $X$.

The form of the cost is inherited by the value process $G^\psi$.
\begin{lemma}\label{lem:cost_decomposition}
We suppose {\normalfont \ref{itm:lower_semicontinuity}, \ref{itm:bounded_stopping}, \ref{itm:Feller}} and the cost is given by $S_t= \Lambda(A_t,X_t)$ as above. Then for any $\psi\in LSC_b(O)$, 
the value process decomposes as
$$
	G_t^\psi=H^\psi(A_t,X_t),
$$
where $H^\psi:\R^d\times O\rightarrow \R$ is the $(A,X)$-r\'{e}duite of $\psi-\Lambda$.  
\end{lemma}
\begin{proof}
	This lemma is simply a restatement of the definition of $G^\psi$ under the additional structure given by $S_t=\Lambda(A_t,X_t)$. Indeed,
	\begin{align*}	
		G_t^\psi=&\ \sup_{\sigma \in \mathcal{S}_t}\mathbb{E}^{\mathbb{P}^\mu}\big[\psi(X_\sigma)-\Lambda(A_\sigma,X_\sigma)|\mathcal{F}_t\big]\\
		=&\ \sup_{\sigma \in \mathcal{S}^{A_t,X_t}}\mathbb{E}^{A_t,X_t}\big[\psi(X_\sigma)-\Lambda(A_\sigma,X_\sigma)\big],
	\end{align*}		
	which is the definition of the r\'{e}duite of $\psi-\Lambda$.
\end{proof}

We assume that $\Lambda$ satisfies a $(A,X)$-twist condition, namely,  we suppose that:
\begin{enumerate}[label=$\mathbf{D\arabic*}$]\setcounter{enumi}{-1}
\item\label{itm:twist}  For $\sigma\in \mathcal{S}^{a,x}$, the equation
\begin{align}\label{eqn:twist}
	\mathbb{E}^{a,x}\big[\nabla_a \Lambda(A_\sigma,X_\sigma)\big]=\nabla_a\Lambda (a,x)
\end{align}
implies $\sigma=0$.
\end{enumerate}

If $\psi$ is continuous then $H^\psi$ is the continuous viscosity solution of the quasivariational inequality:
\begin{align}\label{eqn:quasivariational_inequality}
	\max\big\{\Delta_{a,x} H^\psi(a,x),\psi(x)-\Lambda(a,x) -H^\psi(a,x) \big\}\leq 0.
\end{align}

Rather than giving details on the processes, we make an assumption directly on solutions of (\ref{eqn:quasivariational_inequality}):
\begin{enumerate}[label=$\mathbf{D\arabic*}$]
	\item\label{itm:differentiable} For any $\psi\in \mathcal{B}_D$ and $\bar{\mathbb{P}}\in\mathcal{T}(\mu)$ that maximizes
	$$
		\mathbb{E}^{\bar{\mathbb{P}}}\big[\psi(X_T)-\Lambda(A_T,X_T)\big],
	$$
	we have that the map for $h\in \R^d$,
	$$
		h\mapsto H^\psi(A_T+h,X_T),
	$$
	is  differentiable $\bar{\mathbb{P}}$-almost surely.  
	We also suppose that the stopping time given by
	$$
		\tau^*=\inf\big\{t;\ H^\psi(A_t,X_t)=\psi(X_t)   - \Lambda(A_t, X_t)  \big\},
	$$
	satisfies
	$$
		H^\psi(A_{\tau^*},X_{\tau^*})=\psi(X_{\tau^*})   - \Lambda(A_{\tau^*}, X_{\tau^*}), \ \ \mathbb{P}^\mu\ {\rm almost\ surely}.
	$$
\end{enumerate}

Examples:
\begin{itemize}
	\item The Lagrangian case is when $A_t=t$.  We set $L_\Lambda(t,x)=\Delta_{a,x}\Lambda(t,x) = \partial_t \Lambda(t,x)+\Delta\Lambda(t,x)$ to be the Lagrangian.  Then $\Lambda$ is $(A,X)$-twisted if  $t\mapsto L_\Lambda(t,x)$ is either strictly increasing or decreasing because
	\begin{align*}
	\mathbb{E}^{t,x}\big[\partial_t \Lambda(t+\sigma,X_\sigma)\big]-\partial_t\Lambda (t,x)=\mathbb{E}^{t,x}\Big[\int_t^\sigma \partial_t L_\Lambda(t+r,X_r)dr\Big]
\end{align*}
 is either strictly positive or strictly negative if $\sigma\not=0$.
  This has been studied in \cite{GKPS} for the case when $X_t=W_t$ is $d$-dimensional Brownian motion.  In the case that $t\mapsto L_\Lambda$ is decreasing to obtain the result we must assume that $\mu$ and $\nu$ are disjoint otherwise \ref{itm:differentiable} would fail.

	\item The recent work \cite{gassiat2019free} provides a manner to generalize the previous example to the case where $A$ is  an additive function of $X$, and $a\mapsto L_\Lambda(a,X)$ is strictly increasing.

	\item Considering costs where $A_t=X_0$ and thus $S_t=c(X_0,X_t)$ generalizes the study in \cite{GKP-Monge} where $X_t=W_t$ is $d$-dimensional Brownian motion.
\end{itemize}
 Here, we list additional possible cases:
\begin{itemize}
	\item The previous cases can be mixed with $A=(t,X_0)$.  This makes it easier to satisfy \ref{itm:twist}, although it may be difficult to check differentiability {\normalfont \ref{itm:differentiable}} in general.

	\item Suppose $(A,X)$ is generated by a uniformly elliptic operator, and suppose that $a\mapsto \Delta_{a,x} \nabla_a\Lambda(a,x)$ is strictly monotone.  Then \ref{itm:twist} holds, and we expect differentiability {\normalfont \ref{itm:differentiable}}  from elliptic regularity.

	\item (Possible Example) Taking the process $A_t=\sup_{s\in [0,t]}\{X_s\}$ possibly generalizes the Az\'{e}ma-Yor embedding ($\Lambda(a,x)=a$ and $X_t=W_t$ is one-dimensional Brownian motion).  It is clear that 
	\ref{itm:twist}  holds if $a\mapsto \Lambda(a,x)$ is increasing and strictly concave or convex.  However, satisfying \ref{itm:differentiable} is highly nontrivial in this case, and the assumption \ref{assumptionS} on the cost will not hold, so more work is needed to understand these problems.
\end{itemize}

We now state and prove our final theorem.

\begin{theorem}\label{thm:general_barrier}
	We suppose all the assumptions of the paper, {\normalfont \ref{itm:lower_semicontinuity}, \ref{itm:bounded_stopping}, \ref{itm:Feller}, \ref{itm:submartingale}, \ref{assumptionPoincare}-\ref{assumptionMeasures}}, and in particular {\normalfont \ref{itm:twist}} that $S$ is $(A,X)$-twisted and {\normalfont \ref{itm:differentiable}} hold.  Then there is a unique minimizer to $\mathcal{P}_S(\mu,\nu)$ given by
	$$
		\tau^*=\inf\big\{t;\ H^\psi(A_t,X_t)=\psi(X_t)   -\Lambda(A_t,X_t) \ \big\},
	$$
	where $\psi\in \mathcal{B}_D$ is a dual maximizer.
\end{theorem}
\begin{proof}
	We let $\psi$ be a dual maximizer, cf.\ Theorem \ref{thm:strong_dual}.
	For any $\bar{\mathbb{P}}\in \mathcal{T}(\mu)$ that maximizes 
	\begin{align}\label{eqn:maximize}
		\mathbb{E}^{\bar{\mathbb{P}}}\big[\psi(X_T)-\Lambda(A_T,X_T)\big]
	\end{align}
	we have
	$$
		H^\psi(A_T,X_T)=\psi(X_T)-\Lambda(A_T,X_T)
	$$
	holds $\bar{\mathbb{P}}$ a.s.\ by the dynamic programming principle of Theorem \ref{thm:verification}.
	Since $H^\psi(a,x)\geq \psi(x)-\Lambda(a,x)$ it follows that
	$$
	\nabla_a H^\psi\big(A_T,X_T\big)=-\nabla_a\Lambda\big(A_T,X_T\big)
	$$
	at points of differentiability, which occur $\bar{\mathbb{P}}$ a.s.\ by \ref{itm:differentiable}.  Also from \ref{itm:differentiable} we have that 
	$$
		H^\psi(A_{\tau^*},X_{\tau^*})=\psi(X_{\tau^*})-\Lambda(A_{\tau^*},X_{\tau^*}),
	$$
	and $\tau^*\leq T$ holds $\bar{\mathbb{P}}$ almost surely so since $H^\psi(A_t,X_t)$ is a supermartingale,
	$$
		\psi(X_{\tau^*})-\Lambda(A_{\tau^*},X_{\tau^*})\geq   \mathbb{E}^{\bar{\mathbb{P}}}\big[\psi(X_{T})-\Lambda(A_{T},X_{T})\big|\bar{\mathcal{F}}_{\tau^*}\big],
	$$
	and $\tau^*$ is also a maximizer of (\ref{eqn:maximize}). We also have that from the supermartingale property of $H^\psi$, 
	$$
		H^\psi(A_{\tau^*}+h,X_{\tau^*})\geq \mathbb{E}^{\bar{\mathbb{P}}}\big[H^\psi(A_{T}+h,X_{T})\big|\bar{\mathcal{F}}_{\tau^*}\big],
	$$
	and  equality holds at $h=0$. 
 Therefore,  taking a derivative by \ref{itm:differentiable}, we get $\bar{\mathbb{P}}$-almost surely
 \begin{align*}
		-\nabla_a\Lambda\big(A_{\tau^*},X_{\tau^*}\big)=&\ \nabla_aH^\psi\big(A_{\tau^*},X_{\tau^*}\big)\\
		=&\ \mathbb{E}^{\bar{\mathbb{P}}}\big[ \nabla_a H^\psi(A_{T},X_{T})\big|\bar{\mathcal{F}}_{\tau^*}\big]=\mathbb{E}^{\bar{\mathbb{P}}}\big[- \nabla_a\Lambda(A_{T},X_{T})\big|\bar{\mathcal{F}}_{\tau^*}\big].
	\end{align*}
	
	It then follows from \ref{itm:twist} that any such maximizer is given by $T=\tau^*$.  

	Since the optimal stopping time $\bar{\mathbb{P}}^*\in \mathcal{T}(\mu,\nu)$ to $\mathcal{P}_S(\mu,\nu)$ is a maximizer of (\ref{eqn:maximize}) by Theorem \ref{thm:verification}, it  is uniquely given by $\tau^*$.
\end{proof}

\appendix

\section{Recurrent Processes}\label{sec:recurrent}

We will repeat the results of our paper under alternate assumptions for ergodic processes.  Under these assumptions there is no cemetery state, and instead of \ref{itm:bounded_stopping} we require an assumption of coercivity of the cost.
	\begin{enumerate} [label=\textbf{(A\arabic*')}] \setcounter{enumi}{0}
		\item \label{itm:cost_coercive} We suppose that for any $\bar{T}\geq 0$, $S$ is uniformly integrable over stopping times $\tau\leq \bar{T}$, and
		$$
			\liminf_{\bar{T}\rightarrow \infty}\inf_{\tau\in \mathcal{S},\ \mathbb{E}[\tau]\geq\bar{T}} \mathbb{E}\big[S_{\tau}\big]=+\infty.
		$$
	\end{enumerate}

	We first recover Lemma \ref{lem:value_process}, the proof is similar but we mention the details that change.
	\begin{lemma}\label{lem:value_process_recurrant}
		We suppose {\normalfont \ref{itm:lower_semicontinuity}} and {\normalfont \ref{itm:cost_coercive}}.  For every $\psi\in LSC_b(O)$, we have that $(\psi,M^\psi)\in \mathcal{A}_S$ and
		\begin{align}\label{eqn:concave_maximization_recurrant}
			D_S(\mu,\nu) = \sup_{\psi\in {LSC_b(O)}}\Big\{ \int_{O}\psi(y)\nu(dy)-\mathbb{E}^{\mathbb{P}^\mu}\big[G_0^\psi\big]\Big\}.
		\end{align}
		In particular, $\mathbb{P}^\mu$-almost surely  $G_\sigma^\psi \leq M_\sigma$ for all $(\psi,M)\in \mathcal{A}_S$ and $\sigma \in \mathcal{S}$.
	\end{lemma}

	\begin{proof}
		When verifying the properties of $G^\psi$ we must first cut off at a finite time so that the cost is uniformly integral.  We introduce the approximation for $\bar{T}\geq t$,
$$
	G^{\psi,\bar{T}}_t=\sup_{\sigma\in \mathcal{S}_t}\Big\{\mathbb{E}^{\mathbb{P}^\mu}\big[\psi(X_{\sigma\wedge \bar{T}})-S_{\sigma\wedge \bar{T}}]\big]\Big\}.
$$
It is clear from the argument of Lemma \ref{lem:value_process} that $G^{\psi,\bar{T}}$ is a regular supermartingale and $G^{\psi,\bar{T}}_t\geq \psi(X_t)-S_t$ for $t \le \bar T$.  We clearly have that $G^{\psi}\geq G^{\psi^i,\bar{T}}$ since $\sigma\wedge \bar{T}$ is an admissible stopping time.   For $\sigma\in \mathcal{S}$, there is $\tau\geq \sigma$ that attains the value of $G^\psi$ using assumption \ref{itm:cost_coercive}, such that
	$$
		\mathbb{E}^{\mathbb{P}^\mu}\big[G^\psi_\sigma\big] = \mathbb{E}^{\mathbb{P}^\mu}\big[\psi(X_\tau)-S_\tau\big].
	$$
	Then we have that
	$$
		\E^{\mathbb{P}^\mu}\big[G^\psi_\sigma-G^{\psi^i,\bar{T}}_\sigma\big]\leq \E^{\mathbb{P}^\mu}\big[\psi(X_\tau)-\psi^i(X_{\tau\wedge \bar{T}})-S_\tau+S_{\tau\wedge \bar{T}}\big],
	$$
	which converges to zero as $i\rightarrow \infty$ and $\bar{T}\rightarrow \infty$ by \ref{itm:lower_semicontinuity} and the dominated convergence theorem.  The remainder of the proof is the same as Lemma \ref{lem:value_process}.
 	\end{proof}

 	We continue to adapt the proof of Theorem \ref{thm:weak_duality} with \ref{itm:cost_coercive}.

 		\begin{theorem}\label{thm:weak_duality_recurrant}
		We suppose {\normalfont\ref{itm:lower_semicontinuity}} and {\normalfont\ref{itm:cost_coercive}}.
		 If $\mathcal{P}_S(\mu,\nu)<+\infty$,
		$$
			\mathcal{D}_S(\mu,\nu) = \mathcal{P}_S(\mu,\nu),
		$$
		and  there is $\bar{\mathbb{P}}^*\in \mathcal{T}(\mu,\nu)$ such that $\mathcal{P}_S(\mu,\nu) = \mathbb{E}^{\bar{\mathbb{P}}^*}[S_{T}]$.
	\end{theorem}
	\begin{proof}
	The proof is identical to that of Theorem \ref{thm:weak_duality}, except that the set $\mathcal{T}(\mu,\nu)$ is no longer compact.  However, \ref{itm:cost_coercive} implies that if $\mathcal{P}_S(\mu,\nu)<+\infty$ then we can restrict to $\bar{\mathbb{P}}\in \mathcal{T}(\mu,\nu)$ with $\mathbb{E}^{\bar{\mathbb{P}}}[T]\leq \bar{T}$ for a constant $\bar{T}$, which is compact.  This implies in particular that $\nu \mapsto  P_S(\mu,\nu)$ is lower-semicontinuous and that if $P_S(\mu,\nu)<+\infty$ then the minimum is attained.
	\end{proof}

We also repeat the following `verification' type result for the dual optimizer. 
 	\begin{theorem}\label{thm:verification_recurrant} 
 		Suppose {\normalfont \ref{itm:lower_semicontinuity}} and {\normalfont\ref{itm:cost_coercive}} and that $\psi\in LSC_b(O)$ attains the maximum of $\mathcal{D}_S(\mu,\nu)$, and $\bar{\mathbb{P}}^*\in \mathcal{T}(\mu,\nu)$ minimizes {\normalfont (\ref{eqn:primal})}.  Then $\bar{\mathbb{P}}^*$ maximizes
		\begin{align}\label{eqn:auxilliary_recurrant}
			\E^{\bar{\mathbb{P}}}\big[\psi(X_T)-S_T\big]
		\end{align}
		over $\bar{\mathbb{P}}\in \mathcal{T}(\mu)$.

		Furthermore, for any maximizer $\bar{\mathbb{P}}\in \mathcal{T}(\mu)$ of {\normalfont (\ref{eqn:auxilliary_recurrant})}, we have
\begin{enumerate}
  \item $G^\psi_{T} = \psi(X_{T})-S_{T}$ holds $\bar{\mathbb{P}}$ almost surely, 
  \item $G^\psi_{t\wedge T}$ is a $(\bar{\Omega},\bar{\mathbb{F}},\bar{\mathbb{P}})$ martingale, i.e., $M^\psi_{t\wedge T} = G^{\psi}_{t\wedge T}$ holds $\bar{\mathbb{P}}$ almost surely for all $t\in \mathbb{R}^+$.
\end{enumerate}
		 	\end{theorem}
 	\begin{proof}
 		The proof is identical to the proof of Theorem \ref{thm:verification}.
 	\end{proof}

We finally repeat Lemma \ref{lem:psi_maximization}.
 	\begin{lemma}\label{lem:psi_maximization_recurrent} 
 		We suppose {\normalfont\ref{itm:lower_semicontinuity}}, {\normalfont\ref{itm:cost_coercive}} and $\psi\in LSC_b(O)$. We let 
 		$$
 			\psi^{\text{\it max}}(y):=\sup_{\phi\in C_b({O})}\big\{\phi(y);\ \phi(X_\sigma(\omega))\leq G_\sigma^\psi(\omega) +S_\sigma(\omega),\ \forall\ \sigma\in \mathcal{S},\ \mathbb{P}^\mu-a.e.\ \omega\big\}.
 		$$  
 		Then we have the following:
 		\begin{enumerate}[label=\roman*.]
 			\item\label{itm:psi_increases_recurrant} $\psi^{\text{\it max}}(y)\geq \psi(y)$\ for all $y\in O$;
 			\item\label{itm:same_value_process_recurrant} $G^{\psi^{\text{\it max}}}_\sigma= G^\psi_\sigma,\ \forall\ \sigma\in \mathcal{S},\ \mathbb{P}^\mu-a.e.\ \omega$.
 		\end{enumerate}
 	\end{lemma}
 	\begin{proof}
 		The proof is identical to Lemma \ref{lem:psi_maximization}.
 	\end{proof}

 	\subsection{ Dual attainment}
 	We now assume $m=\gamma$ is the invariant distribution $O$ of the process $X_t$.  	
	The invariant measure $\gamma$ satisfies
 	\begin{align}\label{eqn:invariant_measure}
 		\int_O \Delta \psi(x)\gamma(dx)=0
 	\end{align}
 	for all $\psi\in \mathcal{H}_0$.  The results of Section \ref{sec:pointwise_bounds}  hold
	 and Theorem \ref{thm:finite_attainment} follows if we assume the discrete Markov chain has finite recurrent time between any two points.
 We replace assumption \ref{assumptionPoincare}  with the following:
	\begin{enumerate}[label=\textbf{(C\arabic*')}]\setcounter{enumi}{-1}

	\item\label{assumptionPoincare_recurrent} {\bf [Poincar\'e inequality']} $\exists\ C_p>0$ such that
	$$
		\mathcal{E}(u,u)\geq C_p^{-1}\int_O \big|  u(x)\big|^2\gamma(dx)
	$$
	for all $u\in \mathcal{H}$ with $\int_O u(x)\gamma(dx)=0$.
	\end{enumerate}
	Equation (\ref{eqn:invariant_measure}) implies that the superharmonic functions are all constant make the balayage assumption of \ref{assumptionMeasures} trivial.
		 We need a stronger assumption on $\nu$:
 	\begin{enumerate}[label=\textbf{(C\arabic*')}]\setcounter{enumi}{2}
	\item\label{assumptionMeasures_recurrent} We suppose that $\mu\in \mathcal{H}^*$ and that $\frac{d\nu}{d\gamma}\in C_b(O)$.
	\end{enumerate}
	We  also  assume a maximum principle type property:
 	\begin{enumerate}[label=\textbf{(C\arabic*')}]\setcounter{enumi}{3}
	\item\label{assumptionBoundedProcess_recurrent}  We suppose there is a constant $\lambda$ such that if $\Delta u\leq 1$ for $u\in LSC_b$ 
	and  $\int_O u(x)\gamma(dx) =0$ then
	$
		u(x)\geq -\lambda
	$ for all $x\in O$.
	\end{enumerate}

	We list a few examples of ergodic processes:
 	\begin{example}\ 
 	\begin{enumerate}
 		\item $O\subset \R^d$ is open and bounded with smooth boundary.
 		\item $X$ is reflecting Brownian motion. The generator, $\Delta$, is the Laplacian with Neumann boundary conditions, i,e,\ the set $\mathcal{H}_0$ are the functions with $\Delta h \in C(\overline{O})$ and $\nabla h \cdot n = 0$ on $\partial O$, where $n$ is the normal vector.
 		\item The Dirichlet form is
 		$$\mathcal{E}(u,v) = \frac{1}{2}\int_{O} \nabla u(x)\cdot \nabla v(x)\, dx,$$ and $\gamma(dx)=\frac{1}{|O|}dx$ is proportional to Lebesgue measure.
 	\end{enumerate}
 	\end{example}
	\begin{example}\ 
 	\begin{enumerate}
 		\item $O$ is a closed Riemannian manifold with unit volume.
 		\item $X$ is Brownian motion with the generator as the Laplace Beltrami operator
 		\item The Dirichlet form is
 		$$\mathcal{E}(u,v) = \int_{O} g\big(\nabla u(x),\nabla v(x)\big)\, \gamma(dx),$$
 		where $\gamma$ is the volume form. 
 	\end{enumerate}
 	\end{example}

Here is a possible additional case:
 	\begin{example}\ 
 	\begin{enumerate}
 		\item $O= \R^d$.
 		\item $X_t$ is Brownian motion with confining potential $V$ that is smooth and coercive (i.e., the Ornstein-Uhlenbeck process for $V(x)=\frac{1}{2}|x|^2$). The generator, $\Delta$, is the Laplacian with  drift, i,e,\ 
 		$$
 			\Delta h = \sum_{i=1}^d \frac{\partial^2 h}{\partial x_i^2}-\nabla V\cdot \nabla h.
 		$$
 		\item The Dirichlet form is
 		$$\mathcal{E}(u,v) = \int_{O} \nabla u(x)\cdot \nabla v(x)\, \gamma(dx),$$
		 where $\gamma(dx)=m(dx)=Ce^{-V(x)}dx$. 
 		\item  This example violates {\normalfont\ref{assumptionL}} and {\normalfont \ref{assumptionBoundedProcess_recurrent}}, and would require more careful handling of the behavior as $|x|\rightarrow \infty$.
 	\end{enumerate}
 	\end{example}

 	When we address dual attainment in general, we will not be able to use the upper bound as we did in Section \ref{sec:dual_attainment}.  To circumvent  this we will use a truncation procedure by defining
 	$$
 		\psi^M(x)=\min\big\{\psi(x),M\big\}.
 	$$
 	It is clear that if $\psi$ is lower semicontinuous, bounded below, and a viscosity supersolution, then so is $\psi^M$.  
 	We give an analogy to the space $\mathcal{B}_D$.
 	\begin{definition}\label{def:B_D_recurrent}
 		We say that $\psi\in \mathcal{B}_D'$, if the following properties hold:
 		\begin{enumerate}
 			\item $\psi$ is lower semi-continuous and $\psi^M\in \mathcal{H}$ for all $M$.
 			\item $\int_O \psi(x)\gamma(dx)=0$.
 			\item $\Delta \psi(x)\leq D$ in the sense of viscosity.
 		\end{enumerate}
 	\end{definition}

 	We define the `weak' topology on $\mathcal{B}_D'$ to be the topology of weak convergence in $\mathcal{H}$ for $\psi^M$ for all $M$.  Assumption \ref{assumptionBoundedProcess_recurrent} is necessary for this space to be compact.
	If $\psi \in \mathcal{B}_D'$ then $\psi \geq -D\lambda$ so  $\|\psi^M\|_{L^1(O,\gamma)}$ is uniformly bounded, and
 	\begin{align*}
 		\mathcal{E}(\psi^M,\psi^M)=&\ \int_O \Delta \psi^M(x) \big(M-\psi^M(x)\big)\gamma(dx)\\
 		=&\ D\, M\int_O\gamma(dx)+ \lambda D^2.
 	\end{align*}
	By the uniform bound above, $\mathcal{B}_D'$ is compact with this topology. 

 	We make use of a second regularization of the process by introducing a killing term with rate $\beta>0$.  
	 As $X_t$ is generated by the Dirichlet form $\mathcal{E}(u, v)$, the modified process with the killing rate $\beta$, is generated by the Dirichlet form $\mathcal{E}^\beta (u,v)$ defined as  
 	$$
 		\mathcal{E}^\beta(u,v)=\mathcal{E}(u,v)+\beta\int_O u(x)\, v(x)\, \gamma(dx).
 	$$
 For $X$ that satisfies \ref{assumptionPoincare_recurrent}, $X^\beta$ satisfies \ref{assumptionPoincare}.  We let $S^\beta$ be the cost that is left-continuous and constant on $\mathfrak{C}$.

	For a probability measure $\sigma$ on $O$ with $\frac{d\sigma}{d\gamma}=s\in C_b(O)$, we let $U^\sigma_\gamma\in \mathcal{H}_0$ denote the potential function that satisfies
 	$$
 		\int_O U^\sigma_\gamma(x)\gamma(dx)=0
 	$$
 	and
 	\begin{align}\label{eq:Usigma-gamma}
 		\Delta U^\sigma_\gamma(x) = 1-s(x).
 \end{align}

 	We now give an analogy of Proposition \ref{prop:compact_semicontinuous}.
 	\begin{proposition}\label{prop:compact_semicontinuous_recurrent}
		We suppose {\normalfont \ref{itm:lower_semicontinuity}, \ref{itm:cost_coercive}, \ref{itm:Feller}, \ref{itm:submartingale}, \ref{assumptionPoincare_recurrent}, \ref{assumptionL}, \ref{assumptionS},  and \ref{assumptionMeasures_recurrent}}. 
		 The map $$\psi\mapsto U(\psi) := \int_{{O}}\psi(y)\nu(dy)-\mathbb{E}^{\mathbb{P}^\mu}\big[G_0^\psi\big]$$ is concave and upper-semicontinuous on $\mathcal{B}_{D}'$ with the weak topology.
	\end{proposition}
	\begin{proof}
		Concavity and upper-semicontinuity follow from the structure as the supremum over linear functionals.    We first note that 
		$$\E^{\mathbb{P}^\mu}[G_0^\psi] = \sup_{M\geq 0}\sup_{\beta> 0}\sup_{\bar{\mathbb{P}}\in \mathcal{T}^\beta(\mu)} \E^{\bar{\mathbb{P}}}[\psi^M(X^\beta_{T})-S^\beta_{T}],$$  
		 where $\mathcal{T}^\beta(\mu)$ is the set of stopping times of the process $X^\beta_t$, starting from the distribution $\mu$.
		The inequality $\geq$ is obvious.  For the other inequality we note that because \ref{itm:cost_coercive} the stopping time that achieves the value of $\E^{\mathbb{P}^\mu}[G_0^\psi]$ has finite expectation and thus is approximated well when $\beta$ is small.

		By exactly the same reason as in the proof of Proposition~\ref{prop:compact_semicontinuous},
		$$
			u\mapsto \E^{\bar{\mathbb{P}}}[u(X^\beta_T)]
		$$
		is a continuous linear functional on $\mathcal{B}_{D}'$ for any $\bar{\mathbb{P}}\in \mathcal{T}^\beta(\mu)$. 
		This shows that the map $ \mathcal{B}'_D \ni \psi \mapsto \E^{\bar{\mathbb{P}}}[\psi^M(X^\beta_{T})-S^\beta_{T}]$ is continuous, thus the map $\psi \mapsto \E^{\mathbb{P}^\mu}[G_0^\psi]$ is upper continuous in $\mathcal{B}'_D$.
		 Finally, continuity of $$\psi\mapsto \int_O \psi(x)\nu(dx) =- \int_O \Delta \psi (x) U^\nu_\gamma (x)\gamma(dx)$$ follows from		 \ref{assumptionMeasures_recurrent}.
		\end{proof}

\begin{proposition}\label{prop:dual_normalization_recurrent}
 		We suppose {\normalfont \ref{itm:lower_semicontinuity}, \ref{itm:cost_coercive}, \ref{itm:Feller}, \ref{itm:submartingale}, \ref{assumptionPoincare_recurrent}, \ref{assumptionL}, \ref{assumptionS}, and \ref{assumptionMeasures_recurrent}}. 
		Given $\psi\in LSC_b(O)$, we consider $\psi^{\text{\it max}}$ as in  Lemma {\normalfont\ref{lem:psi_maximization_recurrent}}.  Then in the sense of viscosity,
 		 \begin{align}\label{eqn:uniform_psi_bound_recurrent}
 			 \Delta \psi^{\text{\it max}}(y)\le D,
 		\end{align}
 		and $\psi^{\text{\it max}}-\int_O \overline \psi(x) \gamma(dx)\in \mathcal{B}_D'$.

 		Consequentially, 
 		$$
			\mathcal{D}_S(\mu,\nu) = \sup_{\psi\in \mathcal{B}_D'}U(\psi).
		$$
 	 	\end{proposition}
\begin{proof}
The proof is the same as Proposition \ref{prop:dual_normalization}.
\end{proof}

We now restate our main theorem on attainment of the dual problem, which follows immediately from the two preceding propositions.  
	\begin{theorem}\label{thm:strong_dual_recurrent} 
		We assume {\normalfont \ref{itm:lower_semicontinuity}, \ref{itm:cost_coercive}, \ref{itm:Feller}, \ref{itm:submartingale}, \ref{assumptionPoincare_recurrent}, \ref{assumptionL}, \ref{assumptionS},\ref{assumptionMeasures_recurrent}}, and {\normalfont \ref{assumptionBoundedProcess_recurrent}}.  Then there is $\psi^*\in \mathcal{B}_D'$ that maximizes $\mathcal{D}_S(\mu, \nu)$, that is,
		$$
			\int_{{O}}\psi^*(y)\nu(dy) -\E^{\mathbb{P}^\mu}\big[G^{\psi^*}_0\big]=\mathcal{D}_S(\mu,\nu).
		$$
	\end{theorem}
	\begin{proof}
		Again the proof is essentially identical to that of Theorem~\ref{thm:strong_dual}. 
	\end{proof}

	\subsection{When the cost is the expected stopping time}

We now give a counterpart of Proposition~\ref{prop:expected_time_transient}. A similar result has appeared  in \cite[Theorem 4.7]{baxter1976stopping}.
	\begin{theorem}\label{thm:ergodic_duality}
	We assume  {\normalfont \ref{itm:Feller}, \ref{assumptionPoincare_recurrent}, \ref{assumptionL}},  and {\normalfont \ref{assumptionMeasures_recurrent}}. We have that
	\begin{align}\label{eqn:ergodic_duality}
		\inf_{\bar{\mathbb{P}}\in \mathcal{T}(\mu,\nu)} \mathbb{E}^{\bar{\mathbb{P}}}\big[T\big] = \sup_{x\in O}\Big\{U_\gamma^\nu(x)-U_\gamma^\mu(x)\Big\}.
	\end{align}
	If we assume additionally {\normalfont \ref{assumptionBoundedProcess_recurrent}}, then the value of \ref{eqn:ergodic_duality} is finite.
\end{theorem}
\begin{proof}

 	By Proposition \ref{prop:dual_normalization_recurrent}, for the case $S_t=t$ we  can restrict the dual potential to $\psi\in \mathcal{B}_D'$ with $
 		D(x)=1.
 	$
As a consequence of (\ref{eqn:invariant_measure}) we have that $\Delta \psi(x)=1-s(x)$ for $s=\frac{d\sigma}{d\gamma}$ for some probability measure $\sigma$.  We have thus found that the dual problem can be restricted to potential functions, $\psi=U^\sigma_\gamma$ for any probability measure $\sigma$.   Also, as $\Delta \psi \le 1$, the value process is always given in this case by $G_t^\psi=\psi(X_t)-t$.

 	We then have that for each $\psi=U^\sigma_\gamma$,
 	\begin{align*}
 		\int_O \psi(x)\nu(dx)-\mathbb{E}^{\mathbb{P}^\mu}\big[G^\psi_0\big]=&\ \int_O U^\sigma_\gamma(x)\nu(dx)-\int_OU^\sigma_\gamma(x)\mu(dx)\\
 		=&\ \int_O\big(U^\nu_\gamma(x)-U^\mu_\gamma(x)\big)\sigma(dx)
		\\ \le & \  \sup_{x\in O}\Big\{U_\gamma^\nu(x)-U_\gamma^\mu(x)\Big\}.
 	\end{align*}
	 From this \eqref{eqn:ergodic_duality} follows.

 	 Finally, given \ref{assumptionBoundedProcess_recurrent}, we have that $U^\mu_\gamma$ is bounded below and  $U^\nu_\gamma$ is bounded above  (as $U^\nu_\gamma \in C_b$) thus the supremum is bounded.
\end{proof} 	

 	Furthermore, the point where the maximum is attained on the righthand side of (\ref{eqn:ergodic_duality}) defines a halting point, which characterizes the stopping times that minimize the expected time.
This result first appears in \cite[Theorem 5.1]{lovasz1995efficient} and we give a short proof for completeness.
\begin{corollary}\label{cor:Necessary_sufficient}  We assume  {\normalfont\ref{itm:Feller}, \ref{assumptionPoincare_recurrent}, \ref{assumptionL},  \ref{assumptionMeasures_recurrent}}, and {\normalfont \ref{assumptionBoundedProcess_recurrent}}, and 
	we suppose $\bar{\mathbb{P}}\in \mathcal{T}(\mu,\nu)$ and $\bar{x}\in O$. Then we have optimality of $\bar{\mathbb{P}}$ and $\bar{x}$ in {\normalfont (\ref{eqn:ergodic_duality})}
	if and only if   $\bar{\mathbb{P}}$-almost surely the local time of $X_t$ at $\bar{x}$ before $T$ is $0$. $\bar{\mathbb{P}}$ almost surely.
\end{corollary}
\begin{proof}
	We let $\psi^\epsilon=U^{\sigma^\epsilon}_\gamma$  as defined in \eqref{eq:Usigma-gamma} for $\sigma^\epsilon$ a probability measure with $\frac{d\sigma^\epsilon}{d\gamma}=s^\epsilon\in C_b(O)$, approximating $\delta_{\bar{x}}$ as $\epsilon\rightarrow 0$.
	 Then we use the definition of $\psi^\epsilon$ and the generator to obtain
	\begin{align*}
		\mathbb{E}^{\bar{\mathbb{P}}}\big[{T}\big]=&\ \mathbb{E}^{\bar{\mathbb{P}}}\Big[\int_0^{{T}} \big(\Delta \psi^\epsilon(X_t)+s^\epsilon(X_t)\big)dt\Big]\\
		=&\ \int_O \psi^\epsilon(x)\nu(dx)-\int_O \psi^\epsilon(x)\mu(dx)+\mathbb{E}^{\bar{\mathbb{P}}}\Big[\int_0^{{T}} s^\epsilon(X_t)dt\Big].
	\end{align*}
	Taking the limit as $\epsilon\rightarrow 0$ we have
	\begin{align*}
		\mathbb{E}^{\bar{\mathbb{P}}}\big[{T}\big]=&\ U_\gamma^\nu(\bar{x})-U_\gamma^\mu(\bar{x})+\lim_{\epsilon\rightarrow 0}\mathbb{E}^{\bar{\mathbb{P}}}\Big[\int_0^{{T}} s^\epsilon(X_t)dt\Big].
	\end{align*}
	The final term can be identified as the local time of $X_t$ at $\bar{x}$ for $t\leq T$.
	It follows from  Theorem \ref{thm:ergodic_duality} that for the optimizers $\bar{\mathbb{P}}$ and $\bar{x}$, we have zero local time.  Conversely, if the local time is $0$ then equality holds in (\ref{eqn:ergodic_duality}) so that optimality of $\bar{\mathbb{P}}$ and $\bar{x}$ follows.
\end{proof}

\section{Path Monotonicity}

 		The value function and dynamic programming principle is closely related the path-monotonicity principle of \cite{beiglboeck2017optimal}, analogous to the relationship between convex functions and cyclic-monotonicity.  Indeed, the dual attainment and verification results of our paper recover the result that the support of the minimizing stopping time satisfies the path monotonicity principle. More precisely, the set 
		$$R=\{(\omega,t);\ \psi(X_t(\omega))=G_t^\psi(\omega)- S_t (\omega) \}$$ 
		satisfies the path-monotonicity property; Definition {\normalfont 1.5} of \cite{beiglboeck2017optimal}.   The following is an extension of \cite[Theorem B.1]{GKP-Monge}, where it was proved for the case $X_t=B_t$ the Brownian motion. 

 		To show that $R$ satisfies the path-monotonicity property, we prove that there is no a stop-go pair in the sense of Definition 1.4 of \cite{beiglboeck2017optimal}. For this, we suppose there  is a path $\omega_1$ that continues optimally at $(\omega_1,t_1)$ (i.e.,\ $(\omega_1,t_1)\in R^<$ in the notation of \cite{beiglboeck2017optimal}), that is, there is a stopping-time $\sigma\in \mathcal{S}_t$, with   $\sigma\not=0$, so that
 		\begin{align*}
 			\mathbb{E}^{\mathbb{P}^\mu}\big[S_{t_1+\sigma}\big|\mathcal{F}_{t_1}\big](\omega_1)
 			=&\ -G_{t_1}^\psi(\omega_1)+\mathbb{E}^{\mathbb{P}^\mu}\big[\psi(X_{t_1+\sigma})\big|\mathcal{F}_{t_1}\big](\omega_1),
 		\end{align*}
 		 and another pair $(\omega_2,t_2)\in R$ that stops optimally so that
 		\begin{align*}
 			S_{t_2}(\omega_2)=&\ -G_{t_2}^\psi(\omega_2)+\psi(X_{t_2}(\omega_2)), 
 		\end{align*}
		and $$X_{t_1}(\omega_1)=X_{t_2}(\omega_2).$$
 		On the other hand, from the definition of $G^\psi_t$ we have the inequalities
 		\begin{align*}
 			S_{t_1}(\omega_1)\geq& -G^\psi_{t_1}(\omega_1)+\psi(X_{t_1}(\omega_1))
 		\end{align*}
 		and
 		\begin{align*}
 			\mathbb{E}^{\mathbb{P}^\mu}\big[S_{t_2+\sigma}\big|\mathcal{F}_{t_2}\big](\omega_2)
 			\geq& -G^\psi_{t_2}(\omega_2)+\mathbb{E}^{\mathbb{P}^\mu}\big[\psi(X_{t_2+\sigma})\big|\mathcal{F}_{t_2}\big](\omega_2).
 		\end{align*}
 		Notice that  from the Markov property of $X_t$ and $X_{t_1}(\omega_1)=X_{t_2}(\omega_2)$, 		
\begin{align*}
 \mathbb{E}^{\mathbb{P}^\mu}\big[\psi(X_{t_1+\sigma})\big|\mathcal{F}_{t_1}\big](\omega_1) -
 \psi(X_{t_1}(\omega_1)) = 
 \mathbb{E}^{\mathbb{P}^\mu}\big[\psi(X_{t_2+\sigma})\big|\mathcal{F}_{t_2}\big](\omega_2) - 
 \psi(X_{t_2}(\omega_2)).
\end{align*}
		Combining all these
		we get that
 		$$
 		\mathbb{E}^{\mathbb{P}^\mu}\big[S_{t_1+\sigma}\big|\mathcal{F}_{t_1}\big](\omega_1)+S_{t_2}(\omega_2)\leq S_{t_1}(\omega_1)+\mathbb{E}^{\mathbb{P}^\mu}\big[S_{t_2+\sigma}\big|\mathcal{F}_{t_2}\big](\omega_2).
 		$$
 		Since $\sigma\not=0$, this shows that $(\omega_1,t_1)$ and $(\omega_2,t_2)$ cannot be a stop-go pair,
		which implies the path-monotonicity principle for $R$.

\bibliography{OptimalStoppingSubmartingale}
\bibliographystyle{plain}

\end{document}